\documentclass[11pt,a4paper]{amsart}
\usepackage[utf8]{inputenc}
\usepackage[T1]{fontenc}
\usepackage[UKenglish]{babel}
\usepackage[a4paper,margin=1in]{geometry}

\usepackage{amsmath,amsthm,amsfonts,amssymb}
\usepackage{mathrsfs,dsfont}
\usepackage{graphicx}
\usepackage{url}
\usepackage{verbatim}
\usepackage[dvipsnames]{xcolor}

\usepackage{import}
\usepackage{xifthen}
\usepackage{pdfpages}
\usepackage{transparent}

\usepackage{hyperref}
\usepackage{mathtools}

\usepackage{bbm}

\usepackage{marginnote}

\makeatletter
\def\namedlabel#1#2{\begingroup
	#2%
	\def\@currentlabel{#2}%
	\label{#1}\endgroup
}
\makeatother

\DeclarePairedDelimiter\norm{\lVert}{\rVert}%

\theoremstyle{plain}
\newtheorem{theorem}{Theorem}[section]
\newtheorem{corollary}[theorem]{Corollary}
\newtheorem{lemma}[theorem]{Lemma}

\theoremstyle{definition}

\numberwithin{equation}{section}

\renewcommand\labelenumi{\textup{\alph{enumi})}}

\renewcommand\theenumi\labelenumi
\makeatletter\renewcommand{\p@enumii}{}\makeatother 

\renewcommand{\leq}{\leqslant}

\renewcommand{\geq}{\geqslant}

\DeclareMathOperator{\supp}{supp}

\DeclareMathOperator{\In}{I}

\newcommand{\floor}[1]{\left\lfloor #1 \right\rfloor}
\newcommand{\ceil}[1]{\left\lceil #1 \right\rceil}

\newcommand{\R}{\mathds{R}}
\newcommand{\N}{\mathds{N}}

\newcommand{\I}{\mathds{1}}

\begin{document}
	
	\title[Exponential densities and compound Poisson measures]
	{Exponential densities and compound Poisson measures}
	
	\date{\today}
	
	\author[M.~Baraniewicz]{Miłosz Baraniewicz}
	\address[M.~Baraniewicz]{Faculty of Pure and Applied Mathematics\\ Wroc{\l}aw University of Science and Technology\\ ul. Wybrze{\.z}e Wyspia{\'n}skiego 27, 50-370 Wroc{\l}aw, Poland}
	\email{mioszba@gmail.com}
	
	\author[K.~Kaleta]{Kamil Kaleta}
	\address[K.~Kaleta]{Faculty of Pure and Applied Mathematics\\ Wroc{\l}aw University of Science and Technology\\ ul. Wybrze{\.z}e Wyspia{\'n}skiego 27, 50-370 Wroc{\l}aw, Poland}
	\email{kamil.kaleta@pwr.edu.pl}
	
	\thanks{Research supported by National Science Centre, Poland, grant no.\ 2019/35/B/ST1/02421}

		\begin{abstract}
	We prove estimates at infinity of convolutions $f^{n\star}$ and densities of the corresponding compound Poisson measures for a class of radial decreasing densities on $\R^d$, $d \geq 1$, which are not convolution equivalent. Existing methods and tools are limited to the situation in which the convolution $f^{2\star}(x)$ is comparable to initial density $f(x)$ at infinity. We propose a new approach which allows one to break this barrier. We focus on densities which are products of exponential functions and smaller order terms -- they are common in applications. In the case when the smaller order term is polynomial estimates are given in terms of the generalized Bessel function. Our results can be seen as the first attempt to understand the intricate asymptotic properties of the compound Poisson and more general infinitely divisible measures constructed for such densities. 
	
	\medskip
	
\noindent
\emph{Key-words}: multivariate density, radial decreasing function, compound Poisson measure, subexponential and convolution equivalent distribution, exponential decay, asymptotics, Wright function.

\medskip

\noindent
2010 {\it MS Classification}: 60E05, 60G50, 60G51, 26B99, 47G10, 62H05.\\
	\end{abstract}

	\maketitle

	\section{Introduction}\label{sec:intro}
	
\subsection*{Formulation of the problem and presentation of results}	
By \emph{densities} we understand non-negative, non-zero functions $f \in L^1(\R^d)$, $d \geq 1$ (we assume that $\R^d$ is equipped with the Lebesgue measure). The $n$-fold convolution of a density $f$ is defined inductively as
\begin{align} \label{eq:convolutions}
	f^{1\star} := f \qquad \text{and} \qquad f^{n\star}(x) :=\int_{\R^d} f(x-y)f^{(n-1)\star}(y)dy, \quad n \geq 2.
\end{align}
The compound Poisson measure $P_{\lambda}$ with parameter $\lambda>0$, built for $f$, is a probability measure on $\R^d$ which is given by
\begin{align} \label{eq:compound_dens}
P_{\lambda}(dx) = e^{-\lambda \left\|f\right\|_1} \delta_0(dx) + p_{\lambda}(x)dx \qquad \text{with density} \qquad p_{\lambda}(x) =  e^{-\lambda \left\|f\right\|_1} \sum_{n=1}^{\infty} \frac{\lambda^n f^{n\star}(x)}{n!}.
\end{align}	
It is known that if the density $f$ on $\R^d$ is strictly positive and radial decreasing (or it is just comparable to such a profile) and satisfies
\begin{equation} \label{eq:DSP}
\sup_{|x| \geq 1} \frac{f^{2 \star}(x)}{f(x)} < \infty,
\end{equation}
then there is a constant $c_1>0$ such that
\begin{align} \label{eq:conv_aux_1}
f^{n \star}(x) \leq c_1^{n-1} f(x), \quad |x| \geq 1, \ n \in \N.
\end{align}
This is a special case of \cite[Lemma 2(b)]{KS17} (see \cite[the proof of Lemma 1 and Corollary 3(e)]{KS19}; note that only the condition $K(1)<\infty$, which easily follows from \eqref{eq:DSP}, is needed there). The matching lower estimate
$$
f^{n \star}(x) \geq c_2^{n-1} f(x), \quad |x| \geq 1, \ n \in \N,
$$
which holds with some $c_2>0$, is a direct consequence of radiality and monotonicity of $f$ or its profile, cf.\ Lemma \ref{lem:lower_gen} below. Combined with \eqref{eq:compound_dens}, it yields
\begin{align} \label{eq:conv_aux}
e^{(c_2/2) \lambda} \leq \frac{p_{\lambda}(x)}{\lambda e^{-\lambda \left\|f\right\|_1} f(x)} \leq e^{c_1\lambda}, \quad |x| \geq 1, \ \ \lambda >0.  
\end{align}
A variant of \eqref{eq:DSP} for distributions on half-line and the corresponding estimates have been first studied by Kl\"uppelberg \cite{Klup90}, and Shimura and Watanabe \cite{SW05}. Recently, a version of \eqref{eq:conv_aux_1}, the so-called Kesten's bound, has been obtained for radial densities on $\R^d$ with regular subexponential profiles, see Finkelshtein and Tkachov \cite{FT18}. We also want to mention here the recent contribution of Carlen, Jauslin, Lieb and Loss \cite{CJLL}. That paper does not treat the asymptotic properties of convolutions;\ it investigates the other qualitative properties of solutions to the convolution inequality $f^{2 \star} \leq f$.

We are now in a position to formulate our research problem.
Roughly speaking, if the decay rate of $f$ at infinity is strictly sub-exponential, then \eqref{eq:DSP} always holds; it breaks down for densities $f$ which are products of exponential functions and smaller order terms, see \cite[Lemma 3.2]{KSchi20} and \cite[Proposition 2]{KS17} for formal statements. For example, if 
\begin{align} \label{eq:exp_dens}
f(x) = e^{-m|x|} |x|^{-\gamma} \quad \text{with} \quad m>0, \ \ \gamma \in [0,d),
\end{align}
or
\begin{align} \label{eq:exp_dens_cut}
f(x) = e^{-m|x|} (1 \vee |x|)^{-\gamma} \quad \text{with} \quad m>0, \ \ \gamma \geq 0,
\end{align}
then 
$$
\text{\eqref{eq:DSP} holds} \quad \Longleftrightarrow \quad \gamma > \frac{d+1}{2}.
$$
The main goal of this paper is to explore the behaviour of convolutions $f^{n\star}$ and densities $p_{\lambda}$ for a class of multivariate radial decreasing densities $f(x) = e^{-m|x|} g(x)$, $m>0$, such that
\begin{equation} \label{eq:non-DSP}
\lim_{|x| \to \infty} \frac{f^{2 \star}(x)}{f(x)} = \infty.
\end{equation}
We look for estimates for a large spatial variable. To the best of our knowledge, this is still an open problem. The strongest explicit results will be obtained for densities of the form \eqref{eq:exp_dens}--\eqref{eq:exp_dens_cut} with $\gamma \in [0, (d+1)/2)$ which are the most common in applications. 

Our contributions are divided into three main parts. We discuss each part separately.

\medskip
\noindent
(1) \textsl{Binomial-type upper bound for convolutions:}
In Section 2 we consider a general class of radial decreasing densities which decay at infinity not faster than exponentially and satisfy the doubling condition at zero, see assumption \eqref{A} for precise formulation. We observe that the study of the asymptotic behaviour of $f^{n\star}$'s can be reduced to an analysis of the sequence $(h_n)$ of auxiliary functions which are defined inductively as
$$
				h_1 \equiv \I_{\R^d}, \quad h_{n+1}(x):= \int_{|y-x| < |x|-1 \atop |y| < |x|-1 } \frac{f(x-y)}{f(x)} f(y) h_n(y) dy, \quad n \geq 1,
	      $$ 
see Section \ref{sec:binomial} for details; 				
$h_n$'s can be replaced by the functions $g_n$ which are defined via integrals restricted to the larger set $\left\{y: |y-x|<|x|, |y| <|x| \right\}$. This is more convenient for lower bounds. 
By restricting the domains of integration, we easily see that
	\begin{align}\label{eq:intro_lower_est_1}
\frac{f^{n \star}(x)}{f(x)} \geq g_n(x) \geq h_n(x) \quad \text{and} \quad \frac{f^{n \star}(x)}{f(x)} \geq M_1^{n-1}, \quad n \in \N, \ \ |x| \geq 1, 
  \end{align}
with an explicit constant $M_1$, see Lemma \ref{lem:lower_gen}. We now summarize our results in this part:
\begin{itemize}
\item Theorem \ref{th:th1}, which is the first main result of the paper, states that we have the following \emph{binomial-type upper estimate} 
	\begin{align}\label{eq:intro_est_1} 
		 \frac{f^{n \star}(x)}{f(x)} \leq \sum_{i=0}^n {\binom{n}{i}} M_2^{n-i} h_i(x), \quad n \in \N, \ \ |x| \geq 1,
  \end{align}
	with the constant $M_2$ depending explicitly on $f$. In Corollary \ref{cor:poiss} we translate the above bounds to the corresponding two-sided estimates for the densities $p_{\lambda}$ for $|x| \geq 1$; we also find a counterpart of \eqref{eq:conv_aux} for $|x| < 1$.

  \item Lemma \ref{lem:hn_as_prop} discusses general asymptotic properties of functions $h_n$ and $g_n$ as $|x| \to \infty$. It indicates a fundamental role of $h_2$ and $g_2$ for our analysis. In particular, if $h_2$ (or, equivalently, $g_2$) is bounded, which is also equivalent to \eqref{eq:DSP}, then we recover from \eqref{eq:intro_est_1} the upper estimates of $f^{n \star}(x)$ and $p_{\lambda}(x)$ as in \eqref{eq:conv_aux_1}, \eqref{eq:conv_aux}, see Corollary \ref{cor:general_prop} (a). On the other hand, if $h_2(x) \to \infty$ (or, equivalently, $g_2(x) \to \infty$) as $|x| \to \infty$, which is equivalent to \eqref{eq:non-DSP},
then 
$$
h_n(x), g_n(x)  \to \infty \quad \text{as} \quad |x| \to \infty, \ \text{for every} \ n > 2,
$$
$h_n$ and $g_n$ are asymptotically equivalent at infinity, and
$$
\frac{h_m(x)}{h_n(x)} \to 0 \quad \text{as} \quad |x| \to \infty, \ \text{whenever} \ n > m \geq 1.
$$
This allows us to derive from \eqref{eq:intro_lower_est_1} and \eqref{eq:intro_est_1} that for every $n \in \N$
			$$
			\frac{f^{n\star}(x)}{f(x)} = h_n(x) (1 + o(1)) = g_n(x)(1+o(1)), \quad \text{as} \ \  |x| \to \infty,
			$$ 
see Corollary \ref{cor:general_prop} (b). In particular, for every $1 \leq m < n$ and $\lambda>0$, 
			$$
			\frac{f^{n\star}(x)}{f^{m\star}(x)} \to \infty \quad \text{and} \quad \frac{p_{\lambda}(x)}{f^{m\star}(x)} \to \infty, \quad \text{as} \ \  |x| \to \infty.
			$$
			\end{itemize}
It shows that if \eqref{eq:non-DSP} holds, then the decay rates of $f^{n\star}$ at infinity become slower and slower as $n$ increases to infinity. Consequently, the asymptotic behaviour of $p_{\lambda}(x)$ as $|x| \to \infty$ is much more difficult than that occuring for $f$ satisfying \eqref{eq:DSP}. In order to find it out, we have first to understand the actual contribution of any $f^{n\star}$ to the series defining $p_{\lambda}$. As we already know, this contribution is encoded via $h_n$. 

\medskip
\noindent
(2) \textsl{Analysis of $h_n$ and $g_n$ for exponential densities with doubling terms in higher dimensions:} 
In Section \ref{sec:Hn_Gn} we analyse densities $f$ which take the form $f(x) = e^{-m|x|} g(|x|)$, where $m>0$ and $g$ is a positive and decreasing profile function with the doubling property, see assumption \eqref{B} for precise statement. We simplify the formulas defining the functions $h_n$ and $g_n$ for $d >1$. Observe that for $d=1$ the exponential terms under the integrals defining these auxiliary functions cancel, and we are left to consider the one dimensional integrals of the form
\begin{align} \label{eq:one-d-h-and-g}
h_n(r)=\int_1^{r-1} \frac{g(r-s)}{g(r)} g(s) h_{n-1}(s) ds \quad \text{and} \quad g_n(r) = \int_0^{r} \frac{g(r-s)}{g(r)} g(s) g_{n-1}(s) ds.
\end{align}
In some cases, this simple form allows for direct analysis of the functions $h_n$, $g_n$. On the other hand, if $d>1$, then the structure of integrals defining $h_n$ and $g_n$ is much more complicated. This is related to the fact that in the multivariate case the contribution of the exponential terms $e^{-m|x-y|-m|y|+m|x|}$ under the integral is not negligible and we have to first understand it. 
Here we find the sequences of functions $(H_n)$ and $(G_n)$, defined via integrals over the interval, which generalize \eqref{eq:one-d-h-and-g} -- we show that an extra weight $(s(r-s))^{(d-1)/2}$ appears under the integrals. In Theorems \ref{th:th2} and \ref{th:th4} we show that
\begin{align} \label{eq:small-by-large}
h_n(x) \leq M_3^{n-1} H_n(|x|), \quad g_n(x) \geq M_4^{n-1} G_n(|x|), \quad x \in \R^d \setminus \left\{0\right\}, \ n \in \N,
\end{align}
with explicit constants $M_3, M_4$. The proof of these results is based on elementary calculus, but this is probably the most technical part of the paper. We remark that the upper bound in \eqref{eq:small-by-large} does not require the doubling condition of $g$ at zero.

\medskip
\noindent
(3) \textsl{Final estimates of $f^{n\star}$ and $p_{\lambda}$ for exponential densities with polynomial terms:} Finally, we apply results from parts (1) and (2) to establish estimates for densities defined by \eqref{eq:exp_dens} and \eqref{eq:exp_dens_cut} for $\gamma \in [0,(d+1)/2)$. The borderline case $\gamma = (d+1)/2$ is more complicated and it requires a different approach. It is partly resolved in a forthcoming paper \cite{BK22}. 
 
Our results in Section \ref{sec:appl} can be summarized as follows:
\begin{itemize}
\item First, in Section \ref{sec:direct_1d}, we analyse the densities \eqref{eq:exp_dens} for $d=1$ (i.e.\ $ \gamma \in [0,1)$). Starting from \eqref{eq:one-d-h-and-g} and using induction and the properties of Beta function, we show that
\begin{align} \label{eq:g_n_1d}
g_n(x) =\frac{\Gamma(1-\gamma)^n}{\Gamma((1-\gamma)n)} |x|^{(1-\gamma)(n-1)},  \quad n \geq 1,
\end{align}
see Lemma \ref{lem:direct_1d_gn}. This gives sharp estimates for $f^{n\star} $ and leads to asymptotics
$$
	\lim_{|x| \to \infty} \frac{f^{n\star}(x)}{f(x) |x|^{(1-\gamma)(n-1)}} = \frac{\Gamma(1-\gamma)^n}{\Gamma((1-\gamma)n)}, \quad n \in \N.
	$$
We also observe that the behaviour of the ratio $p_{\lambda}/f$ is described by the generalized Bessel (Wright) function $\phi$. More precisely, we get
	$$
	1 \leq \frac{p_{\lambda}(x)}{ e^{- \norm{f}_1\lambda} e^{-m|x|} |x|^{-1} \phi\big(1-\gamma,0;\Gamma(1-\gamma)\lambda |x|^{1-\gamma}\big) } \leq e^{M_2\lambda}, \quad |x| \geq 1, \ \lambda>0.
	$$
It allows us to derive sharp two-sided explicit estimate for $p_{\lambda}$ for $|x| \geq 1$ and $\lambda>0$ which splits into two different regimes, see Theorem \ref{th:1d} for detailed statement. This is the sharpest result obtained in this paper.

\item For $d>1$ the situation is much more complicated and we cannot hope to compute the functions $h_n$ or $g_n$. However, we are able to find estimates for the functions $H_n$, $G_n$ from \eqref{eq:small-by-large}, see Lemmas \ref{le:le altH_n} and \ref{lem:G_lower}. These bounds are sharp in spatial variable. They seem to be mostly useful in higher dimensions, but formally they are obtained for every $d \geq 1$ -- this is because we want to cover here both examples \eqref{eq:exp_dens} and \eqref{eq:exp_dens_cut}.

The proof of the upper bound for $H_n$ is straightforward -- it follows the steps leading to \eqref{eq:g_n_1d}. This is due to a simple form of this function. On the other hand, the structure of $G_n$ is necessarily more complicated and this causes some extra troubles. 
In order to overcome these obstacles, we first need to find estimates for small arguments and estimate uniformly the incomplete beta function by standard beta function, see Lemma \ref{lem:beta:quotient}. This estimate is critical for the proof -- it is based on an application of the Gauss hypergeometric function.  
	
\item Finally, having the estimates of the functions $H_n$, $G_n$ in hand, we are in a position to give the estimates of convolutions $f^{n \star}$ and the densities $p_{\lambda}$ for both densities \eqref{eq:exp_dens} and \eqref{eq:exp_dens_cut} and any dimension $d \geq 1$, see Theorem \ref{th:poisson}. Again, the estimates of $p_{\lambda}/f$ are given in terms of the generalized Bessel function $\phi$. This leads to the following two-sided estimates of $p_{\lambda}$ for $|x| \geq 1$ and $\lambda >0$ in Corollary \ref{cor:final}: \\
if $\lambda |x|^{\frac{d+1}{2}-\gamma} \leq 1$, then
  $$
			c_1 \leq \frac{p_{\lambda}(x)}{ \lambda e^{- \norm{f}_1 \lambda} e^{-m|x|} |x|^{-\lambda}} \leq c_2;
  $$
if $\lambda |x|^{\frac{d+1}{2}-\gamma} \geq 1$, then
	\begin{equation*}
c_3 \exp\left(c_4 \big(\lambda |x|^{\frac{d+1}{2}-\gamma}\big)^{\frac{1}{\rho_1 + 1}}\right) \leq	 \frac{p_{\lambda}(x) }{ e^{- \norm{f}_1 \lambda} e^{-m|x|} |x|^{-\frac{d+1}{2}}} \leq  c_5 e^{M_2\lambda} \exp\left(c_6 \big(\lambda |x|^{\frac{d+1}{2}-\gamma}\big)^{\frac{1}{\rho_2 + 1}}\right).
	\end{equation*}
Here $c_1,...,c_6$ denote the positive constants depending on $d$ and $\gamma$. We were able to prove this estimate with $\rho_2 = (d+1)/2-\gamma$, and $\rho_1=d-\gamma$ for \eqref{eq:exp_dens} and $\rho_1 = d$ for \eqref{eq:exp_dens_cut}, i.e.\ the sublinear terms in the exponents on both sides have different powers. However, the leading linear term $m|x|$, which comes from the initial density $f$, is sharp. One can conjecture that at least for \eqref{eq:exp_dens} one should have $\rho_1= \rho_2 = (d+1)/2-\gamma$. This is true for $d=1$, but for higher dimensions it is not available yet.
\end{itemize}
	
\subsection*{Beyond the convolution equivalent class -- motivations, long-term goals, applications}	
One of our long-term goals would be to understand the asymptotic behaviour of densities which are not convolution equivalent; estimates proved in this paper are the first step in this direction. 

Recall that the class of univariate \emph{convolution equivalent} densities consists of functions $f:\R \to [0,\infty)$ (strictly positive on $[A,\infty)$ for some $A > 0$) such that there exists $m \geq 0$ such that 
\begin{equation} \label{eq:c1}
        \lim_{x \to \infty} \frac{f(x-y)}{f(x)}  = e^{m y}, \ \ \text{for} \ \ y \in \R, \quad \text{and} \quad \lim_{x \to \infty} \frac{f^{2\star} (x)}{f(x)} = 2 \int e^{my}f(y)dy \textless \infty.
    \end{equation}
	In particular, if $m=0$, then $f$ is called \emph{subexponential density}, see Kl\"uppelberg \cite[Definition 1]{Klup89}.
	The asymptotic behaviour of $f^{n*}(x)$ and $p_{\lambda}(x)$ as $|x| \to \infty$ for univariate convolution equivalent densities is now well understood. More precisely, it is known that \eqref{eq:c1} extends to
\begin{equation} \label{eq:c3}
       \lim_{x \to \infty} \frac{f^{n\star} (x)}{f(x)} = n \left(\int e^{my}f(y)dy \right)^{n-1}, \quad n \in \N,
 \end{equation}
and, consequently, one has
\begin{equation} \label{eq:c4}
\lim_{x \to \infty} \frac{p_{\lambda}(x)}{\lambda f(x)} = \exp \left( \lambda \int \big(e^{my} -1\big) f(y) dy\right), \quad \lambda >0,
 \end{equation}
see e.g.\ \cite[Lemma 3.1 and Theorem 3.2]{Klup89} for densities supported in $[0,\infty)$. 
Finkelshtein and Tkachov proved \eqref{eq:c3} for densities on $\R$ which are weakly subexponential \cite[Theorem 1.1]{FT18}. 
Recently, Watanabe has investigated a subexponential asymptotics for densities of infinitely divisible distributions 
on the half-line \cite{W20}.

Our Theorem \ref{th:1d} goes beyond the setting of convolution equivalent densities and give a counterpart of \eqref{eq:c3} for one-dimensional densities defined by \eqref{eq:exp_dens}. It also provides two-sided estimates (sharp for bounded sets of $\lambda$) for the ratio in \eqref{eq:c4} in that case.

The theory of multivariate convolution equivalent densities have just started to shape up.
Kaleta and Ponikowski proposed in \cite{KP21} the definition of directional convolution equivalent densities and found a useful characterization of this property for almost radial decreasing densities. One implication of this characterization was proved by Kaleta and Sztonyk in \cite{KS19} in a slightly more general setting. The paper \cite{KP21} also includes a multivariate directional variants of \cite[Theorem 3.2]{Klup89} and \cite[Theorem 1]{W20} for almost radial decreasing densities.

Estimates in Theorem \ref{th:poisson} and Corollary \ref{cor:final} give an idea of what kind of behaviors one should expect for multivariate densities which are outside of the convolution equivalent class; it is now fully clear that these behaviours are different. Our results also indicate that it is indeed a challenging problem to get the exact asymptotics for $f^{n\star}$ and $p_{\lambda}$ in that case -- this is related to the fact that the behaviors we identified in this paper essentially depend on the shape of the smaller order terms of $f$. In a sense, it justifies our strategy to look at the special class of examples first.

We remark that subexponentiality and convolution equivalence have been first studied for distributions on halfline, see Chistyakov \cite{Chist64}, Athreya and Ney \cite{AN72}, Chover, Ney and Wainger \cite{CNW73,CNW73b}. They have received attention with applications to branching processes, renewal and queueing theory, random walks, infinitely divisible distributions and L\'evy processes. First results relating these classes of distributions with asymptotic behaviour of the corresponding compound Poisson measures were obtained by Embrechts, Goldie and Veraverbeke \cite{EGV79} ($m=0$), and by Embrechts and Goldie \cite{EG82} ($m \geq 0$). Later on, similar asymptotic problems have been investigated for general infinitely divisible laws on $[0,\infty)$ and $\R$, see e.g.\ Sgibnev \cite{Sg90}, Pakes \cite{P04, P07}, Shimura and Watanabe \cite{SW05}, Watanabe \cite{W08} and Watanabe and Yamamuro \cite{WY10a, WY10b}. 
There is no canonical definition of subexponential and convolution equivalent distributions in higher dimensions -- one can find at least three different approaches in the literature, see Cline and Resnick \cite{CR92}, Omey \cite{O06} (see also Knopova \cite{Knop13} and Knopova and Palmowski \cite{KP19} for applications), Samorodnitsky and Sun \cite{SS16}. 

We expect that our results will find some applications similar to those mentioned above. The densities of the form \eqref{eq:exp_dens} and \eqref{eq:exp_dens_cut} also appear in various problems at the intersection of stochastic processes, semigroups of operators and partial differential equations, including applications in mathematical physics. In the forthcoming paper \cite{BK22} we apply these bounds to establish two-sided estimates of densities for a class of Lamperti stable processes.
	
	\section{General estimates}
	
	In this section we analyze convolutions $f^{n\star}$ through the functions $h_n$, $g_n$. Estimates proven here are basic for our further investigations. 
	Throughout this section we assume that $f$ satisfies the following assumption.
	
		\medskip
	
	\begin{itemize}
\item[\bfseries(\namedlabel{A}{A})] 
    Let $f:\R^d \to (0,\infty]$ be a density such that
    \begin{enumerate}
    \item
    $f$ is radial decreasing function;
    \item
    there is a constant $C_1 \geq 1$ such that $f(x) \leq C_1 f(y)$, for $1 \leq |x| \leq |y| \leq |x|+1$;
    \item
    there is a constant $C_2 \geq 1$ such that $f(x) \leq C_2 f(2x)$, for $|x| \leq 1$.
    \end{enumerate}
\end{itemize}
	
\medskip
	
\noindent
Condition (\ref{A}.b) says that we focus on densities that decay at infinity not too fast. It excludes functions decaying like $\exp\left(-r^\beta\right)$, $\beta>1$, but exponentially and slower decaying functions -- for example $\exp\left(-r^{\beta}\right)$, with $\beta \in (0,1]$, and $r^{-\beta}$, with $\beta>0$ -- are admissible. 
We also require from $f$ some regularity for small arguments -- the condition (\ref{A}.c) will be referred to as the doubling property at $0$. Note, however, that \eqref{A} allows for (integrable) singularities at $0$. 

It is also important to note that all convolutions $f^{n\star}$ inherit from $f$ the property (\ref{A}.a) -- they are radial decreasing functions on $\R^d$, see e.g.\ \cite[p.\ 171]{MR385456}.

	  \subsection{Restricted integrals and the direct lower bound} 	\label{sec:alternative}
	
We set
	$$
	h_1 \equiv g_1 \equiv \I_{\R^d} 
	$$
and define inductively
	$$
	h_{n+1}(x):= \frac{\int_{D(x)} f(x-y) f(y) h_n(y) dy}{f(x)}, \quad  x \in \R^d, \ \ n \in \N,
	$$ 
	$$
	g_{n+1}(x):= \frac{\int_{E(x)} f(x-y) f(y) g_n(y) dy}{f(x)}, \quad  x \in \R^d, \ \ n \in \N.
	$$ 
	where
	\begin{align}\label{eq:D_def}
	D(x) = \left\{z \in \R^d: |z| < |x|-1, |x-z| < |x|-1 \right\},
	\end{align}
	see Figure \ref{fig1}, and 
	\begin{align}\label{eq:E_def}
	E(x) =
	\left\{z \in \R^d: |z| < |x|, |x-z| < |x| \right\}.
	\end{align}
  Note that the set $D(x)$ is an empty set for $|x| \leq 2$.
	Here we use the convention that integral over an empty set is equal to zero and that $a/(+\infty) = 0$ for $a \geq 0$.
	Consequently, $h_n(x)=0$ for $|x|\leq n$, whenever $n \geq 2$.
  Moreover, the functions $h_n$, $g_n$ are radial and 
	\begin{align} \label{eq:h-and-g}
h_n(x) \leq g_n(x), \quad x \in \R^d, \ \ n \in \N.
\end{align}
We remark that the two types of auxiliary functions are introduced here for technical reasons. The functions $g_n$ are more natural and they can be used directly for lower estimates. The functions $h_n$ are needed in proving the upper bounds. Throughout the paper $B_r(x)$ denotes an open Euclidean ball centered at $x$ with radius $r$.

\begin{lemma} \label{lem:lower_gen} Under assumption \textup{(\ref{A}.a)} we have
	\begin{align}\label{eq:lower_gen}
		\frac{f^{n \star}(x)}{f(x)} \geq g_n(x) \vee M_1^{n-1}, \quad  |x| \geq 1, \ \ n \in \N,
		  \end{align}
	where $M_1:= f(1,0,\ldots,0)|B_{1/2}(0)|$.
\end{lemma}
\begin{figure}\centering
	\includegraphics[width = \textwidth]{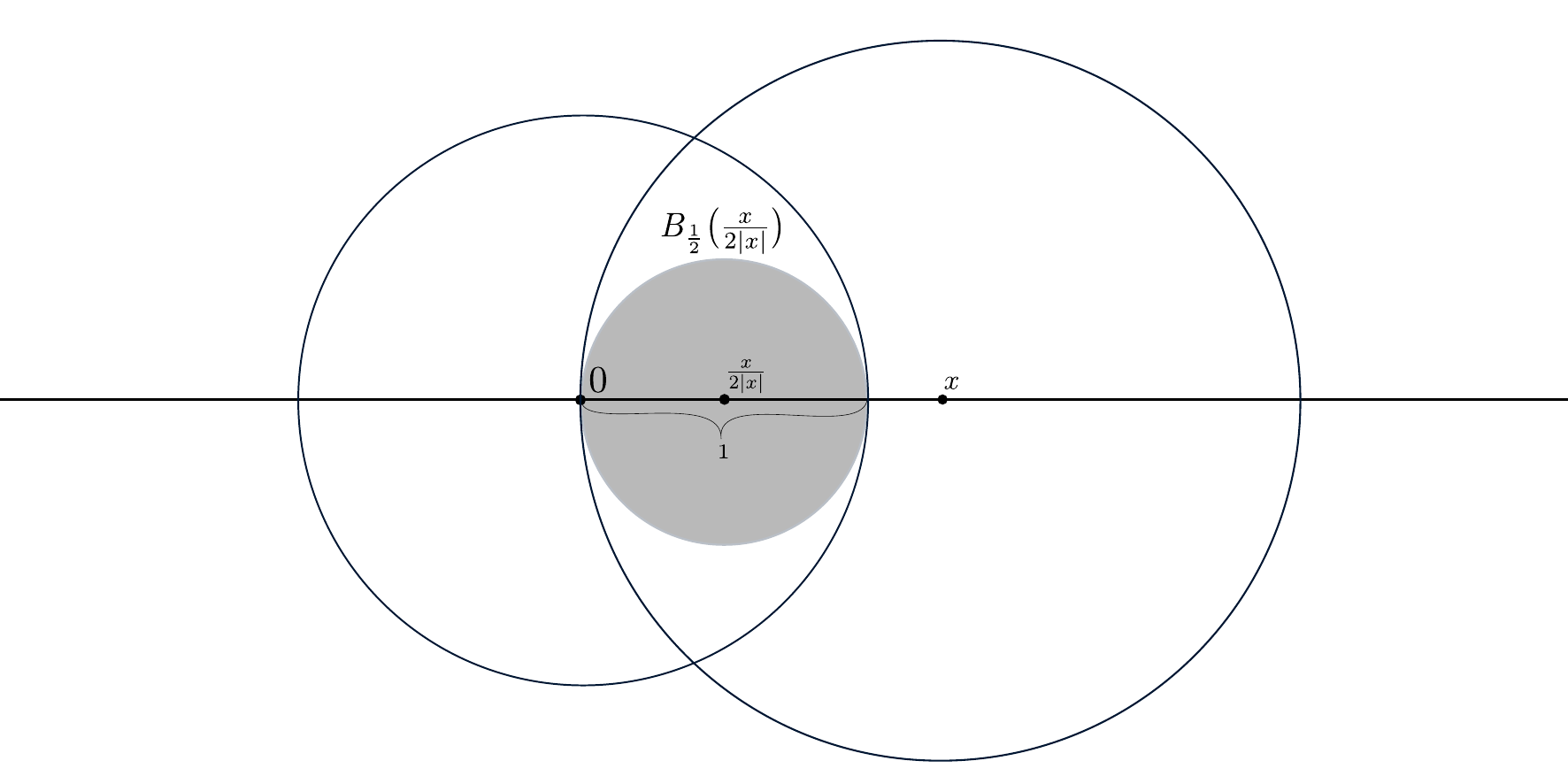}
	\caption{Illustration of inclusion \eqref{balls inclusion}.}\label{fig2}
\end{figure}

\begin{proof} The estimate by $g_n$ is clear as it is given by restricted integrals. We only need to establish the estimate by $M_1^{n-1}$. This is trivial for $n=1$.
 By assumption (\ref{A}.a) and the inclusion
 \begin{equation}\label{balls inclusion}
 	 B_{\frac{1}{2}}\big(\tfrac{x}{2|x|}\big) \subset B_{|x|}(x) \cap B_1(0),
 \end{equation}
 we have
$$
f^{(n+1)\star}(x) \geq \int_{B_{|x|}(x) \cap B_1(0)} f(x-y) f^{n\star}(y)dy 
\geq f(x) \int_{B_{\frac{1}{2}}(\frac{x}{2|x|})} f^{n\star}(y) dy = f(x) \int_{B_{\frac{1}{2}}(x_0)} f^{n\star}(y) dy,
$$
with $x_0=(1/2,0,\dots,0)$, because all convolutions $f^{n\star}$ are radial functions.
Therefore, it suffices to show that 	
  $$
	\int_{B_{\frac{1}{2}}(x_0)} f^{n\star}(y) \geq \left( f(1,0,\ldots,0)|B_{1/2}(0)|\right)^n = M_1^n, \quad n \in \N.
	$$
It is clear for $n=1$, by monotonicity. Similarly,
\begin{align*}
\int_{B_{\frac{1}{2}}(x_0)} f^{n\star}(y) dy
		& \geq \int_{B_{\frac{1}{2}}(x_0)} \int_{\R^d} f(y-z) f^{(n-1)\star}(z) dz dy \\ 
		& \geq \int_{B_{\frac{1}{2}}(x_0)} \int_{B_{\frac{1}{2}}(x_0)} f(y-z) f^{(n-1)\star}(z) dz dy \\
		& \geq f(1,0,\ldots,0) |B_{1/2}(0)| \int_{B_{\frac{1}{2}}(x_0)} f^{(n-1)\star}(z) dz.
	\end{align*}
	The assertion holds by induction.
\end{proof}
		
	\subsection{The upper bound of binomial type} \label{sec:binomial}

	We now state our first main theorem which gives the upper estimate for $f^{n \star}$ in terms of functions $h_n$, $n \in \N$.
	\begin{theorem} \label{th:th1}
		Let assumption \textup{\eqref{A}} hold. Then
			$$
			f^{n \star}(x) \leq \left(\sum_{i=1}^n {\binom{n}{i}} M_2^{n-i} h_i(x)\right)f(x), \quad x \in \R^d, \ n \in \N,
			$$
		with $M_2:= (C_1 \vee C_2) \left\|f\right\|_1$, where the constants $C_1, C_2$ come from \textup{(\ref{A})}.  
	\end{theorem}
	\noindent
	The proof of this result will be given after a sequence of three auxiliary lemmas.

First we observe that the sets $D(x)$ defined by \eqref{eq:D_def} have the following monotonicity property. 
	
	\begin{lemma} \label{lem:monotonicity_dx} 
		For $x \in \R^d$ and $c \geq 1$ we have $D(x) \subseteq D(cx)$.
	\end{lemma}
	
	\begin{proof}
	We only need to consider $c > 1$. For $|x| \leq 2$ we have $D(x) = \emptyset$ and the assertion holds trivially. 
		Let $|x| > 2$ and let $z \in D(x)$. We then have 
$$|z| \leq |x|-1 \leq |cx|-1$$
and 
$$|cx-z| \leq (c-1)|x| + |x-z|\leq (c-1)|x|+ |x|-1=|cx|-1,$$ 
showing that $z \in D(cx)$.
	\end{proof}

	The next lemma shows that the restricted integrals defining $h_n$'s inherit the property (\ref{A}.b) from $f$.  
	
	\begin{lemma} \label{lem:hn_prop} 
		Let \textup{(\ref{A}.a,b)} holds with a constant $C_1 \geq 1$. Then for every $n \in \N$ one has 
		$$
		h_n(x) f(x) \leq C_1 h_n(y) f(y), \quad  1 \leq |x| \leq |y| \leq |x|+1. 
		$$
	\end{lemma}
	
	\begin{proof}
	For $n=1$ this is just (\ref{A}.b). Let $n \geq 2$. Since all $h_n$'s are nonnegative and $h_n(x) = 0$ for $|x| \leq n$, we only need to consider $|x| > n$. Moreover, $h_n$ is radial and we may assume that $x = (x_1,0,\ldots,0)$ and $y=(y_1,0,\ldots,0)$ are such that $n < x_1 \leq y_1 \leq x_1+1$. For $z \in D(x)$ we have 
	$$
	|x-z| \geq |x|-|z| \geq 1,
	$$
	$$
	|x-z| \leq |y-z| \leq |y-x| + |x-z| \leq 1 + |x-z|, 
	$$
	and, by (\ref{A}.b), 
	$$
	f(x-z) \leq C_1 f(y-z), \quad z \in D(x).
	$$
This gives that
	$$
		h_n(x)f(x)=\int_{D(x)} f(x-z) f(z) h_{n-1}(z) dz \leq C_1 \int_{D(x)} f(y-z) f(z) h_{n-1}(z) dz.
	$$ 
 Finally, it follows from Lemma \ref{lem:monotonicity_dx} that we can increase the domain of integration to $D(y)$, getting $h_n(x)f(x) \leq C_1 h_n(y)f(y)$. This is the claimed inequality. 

	\end{proof}
We also need the following lemma which says that the doubling condition around zero implies a certain upper estimate 
	for the convolutions on some neighbourhood of zero.
	
	\begin{lemma} \label{lem:es_for_dob} 
		Let $f:\R^d \to (0,\infty]$ be such that \textup{(\ref{A}.a)} holds. Suppose, in addition, that there exists $0 < b < \infty$ and a constant $C \geq 1$ such that
		$$
		f(x) \leq C f(2x), \quad  |x| \leq b.
		$$
Then for $n\in\N$ we have 
$$
f^{n \star}(x) \leq n (C \norm{f}_1)^{n-1} f(x), \quad |x| \leq 2b.
$$
\end{lemma}

\begin{proof}	
	We use induction. For $n=1$ the assertion holds trivially. Suppose that the claimed bound is true for $n$. We will show that it is also true for $n+1$. Observe first that by (\ref{A}.a), 
	$$
	f^{(n+1) \star}(0) \leq f(0) \int_{\R^d} f^{n\star}(y) dy = f(0) \norm{f}_1^n \leq f(0) (n+1)(C \norm{f}_1)^n,
	$$
	that is the claim holds for $x=0$. 
	
	Let $x \neq 0$ be such that $|x| \leq 2b$ and let $A = \{z:|x-z| >|z|\}$. Note that for $y \in A$ 
		one has $|x-y| > \frac{|x|}{2}$, and for $y \in A^{c}$, $|y| \geq \frac{|x|}{2}$. Combined with (\ref{A}.a), this gives that
		\begin{align*}
		f^{(n+1) \star}(x) & = \int_{A}f(x-y)f^{n\star}(y) dy + \int_{A^c}f(x-y)f^{n\star}(y) dy \\
		&\leq \int_{A}f\left(\frac{x}{2}\right)f^{n\star}(y) dy + \int_{A^c}f(x-y)f^{n\star}\left(\frac{x}{2}\right) dy.
		\end{align*}
		Now, by doubling of $f$ and the induction hypothesis (note that $\frac{|x|}{2} \leq b$), we conclude that
		\begin{align*}
		f^{(n+1) \star}(x) & \leq C f(x) \int_{A} f^{n\star}(y) dy + n (C \norm{f}_1)^{n-1} f\left(\frac{x}{2}\right)\int_{A^c} f(x-y) dy \\
		& \leq C f(x) \norm{f}_1^n + n (C \norm{f}_1)^{n-1} C \norm{f}_1 f(x) \\
		& \leq (C \norm{f}_1)^n f(x) + n (C \norm{f}_1)^n f(x) \\
		&  = (n+1) (C \norm{f}_1)^n f(x).
		\end{align*}
		This completes the proof. 
\end{proof}	

 We will need the following decomposition of $D(x)^c$:
	\begin{align} 
		D(x)^c_{+} & := D(x)^c \cap \left\{z \in \R^d: |x-z| \leq |z|  \right\}, \label{eq:D_plus_def} \\
		D(x)^c_{-} & := D(x)^c \cap \left\{z \in \R^d: |x-z| > |z|  \right\},   \label{eq:D_minus_def}
	\end{align}	
	see Figure \ref{fig1} for illustration. We are now ready to give the proof of Theorem \ref{th:th1}. 
	
		\begin{figure}\centering
			\includegraphics[width = \textwidth]{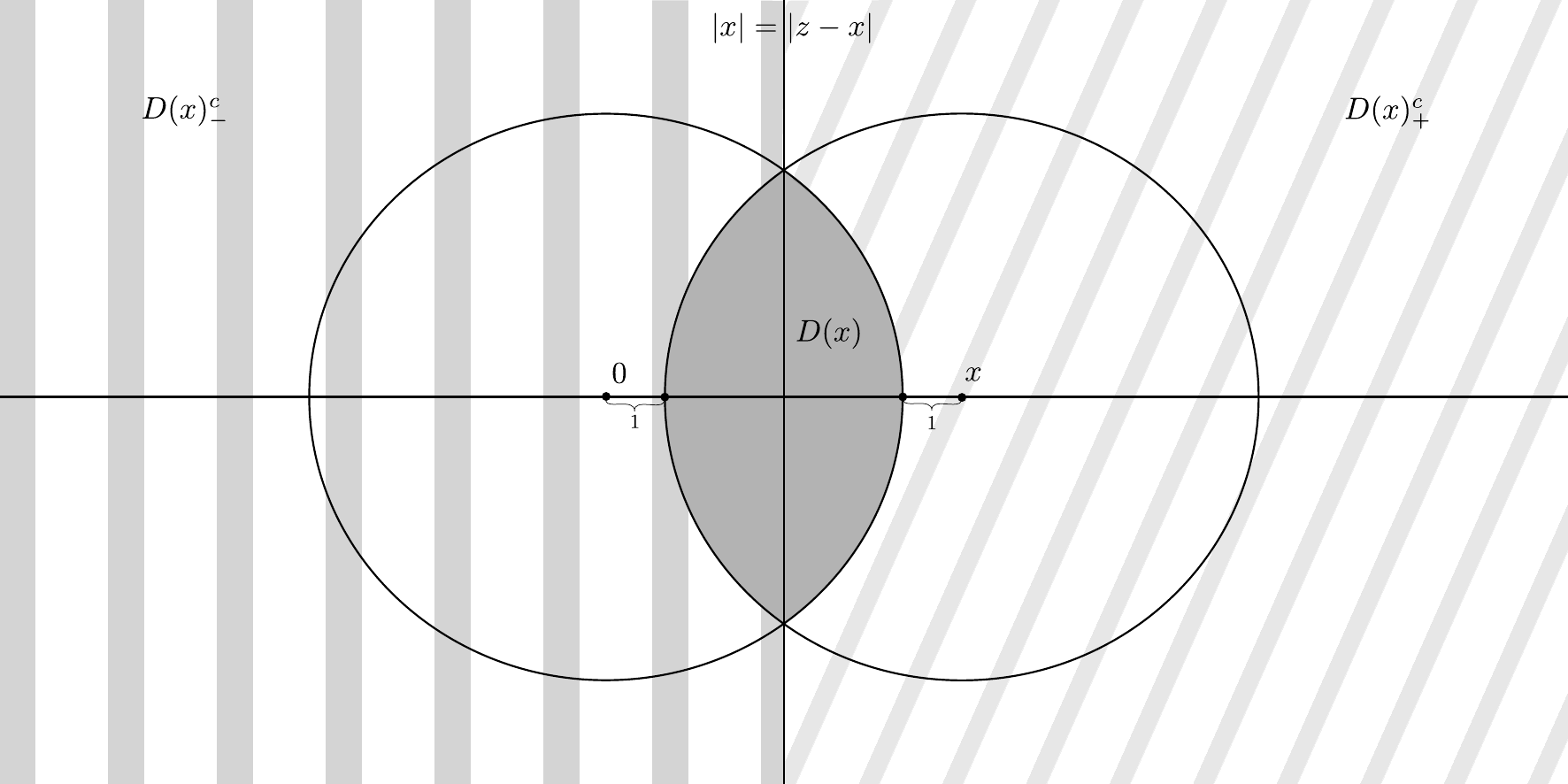}
			\caption{Illustration of the sets $D(x), D(x)_{-}^c$ and $D(x)_{+}^c$ defined by \eqref{eq:D_def}, \eqref{eq:D_plus_def} and \eqref{eq:D_minus_def}, respectively.}\label{fig1}
		\end{figure}

\begin{proof} [Proof of Theorem \ref{th:th1}] 
		For $n = 1$ the estimate holds trivially. We only need to check the induction step. Suppose the assertion holds for a given $n$ and all $x \in \R^d$. We will show that the same is true for $n +1$. 
		
		\smallskip
		
	\noindent
	For $|x| \leq 2$ the claimed bound follows directly from Lemma \ref{lem:es_for_dob} --- this is a consequence of assumption \textup{(\ref{A}.c)} and the fact that $h_n(x)=0$ for $|x| \leq n$ and $n \geq 2$, and $h_1 \equiv \I_{\R^d}$.
		
		\smallskip
		
		\noindent 
		Let $|x| > 2$.
		Since $\int_{\R^d} f^{n \star}(y) dy = \left\|f\right\|_1^n$, we have 
		\begin{align}
		f^{(n+1)\star}(x) & = \left( \int\limits_{D(x)^c_-} + \int\limits_{D(x)^c_+} + \int\limits_{D(x)} \right) 
		f(x-y) f^{n\star} (y) dy \nonumber \\
		&\leq \sup_{z \in D(x)^c_-} f(x-z)\norm{f}_1^{n} + \sup_{z \in D(x)^c_+} f^{n\star}(z) \norm{f}_1  + \int\limits_{D(x)} f(x-y) f^{n\star} (y) dy . \label{eq:th1_aux_1}
		\end{align} 
		Denote $\widetilde x := x - x/|x|$. Clearly, $|\widetilde x| = |x|-1 > 1$. It follows from definitions of the sets $D(x)^c_+$ and $D(x)^c_-$ in \eqref{eq:D_plus_def}--\eqref{eq:D_minus_def} and the fact that the both functions $f$ and $f^{n \star}$ are radial decreasing that
		$$
		\sup_{z \in D(x)^c_-} f(x-z) = \sup_{\left\{z: \, |x-z|=|x|-1 \right\}} f(x-z) = f(\widetilde x)
		$$
and		
		$$
		\sup_{z \in D^c_+(x)} f^{n\star}(z) = \sup_{\left\{z: \, |z|=|x|-1 \right\}}  f^{n\star}(z) = f^{n\star}(\widetilde x).
		$$
		Consequently, the sum on the right hand side of \eqref{eq:th1_aux_1} is equal to	
		\begin{align*}		
		f(\widetilde x )\norm{f}_1^{n} + f^{n \star}(\widetilde x) \norm{f}_1 +  \int\limits_{D(x)} f(x-y) f^{n\star} (y) dy.
		\end{align*}
		By the induction hypothesis, this can be estimated from above by
		\begin{align}	\label{eq:th1_aux_2}
		f(\widetilde x) \norm{f}_1^{n} +  \left(\sum_{i=1}^n {\binom{n}{i}}  M_2^{n-i} h_i(\widetilde x ) \right) f(\widetilde x) \norm{f}_1
		+  \int\limits_{D(x)} f(x-y) \sum_{i=1}^{n}\binom{n}{i}M_2^{n-i} h_i(y) f(y) dy.
		\end{align}
	  Now, since $1 < |\widetilde x| = |x| - 1$, we can use (\ref{A}.b) and Lemma \ref{lem:hn_prop}, to show that 
		$$f(\widetilde x) \leq C_1 f(x)$$ and $$h_i(\widetilde x) f(\widetilde x) \leq C_1 h_i(x)f(x), \quad i=1,\ldots,n.$$ 
	This gives that \eqref{eq:th1_aux_2} is less than or equal to 
	 \begin{align*}
		 C_1 \norm{f}_1^{n} & f(x) +C_1 \norm{f}_1\left(\sum_{i=1}^n {\binom{n}{i}}  M_2^{n-i} h_i(x) \right) f(x) 
		+   \int\limits_{D(x)} f(x-y) \sum_{i=1}^{n}\binom{n}{i}M_2^{n-i} h_i(y) f(y) dy \\
		& \leq M_2^n f(x) +\left(\sum_{i=1}^n {\binom{n}{i}}  M_2^{n-i+1} h_i(x) \right) f(x) 
		+   \int\limits_{D(x)} f(x-y) \sum_{i=1}^{n}\binom{n}{i}M_2^{n-i} h_i(y) f(y) dy.
		\end{align*} 
		
		Changing the last integral with sum and using the definition of the functions $h_i$, we finally get
		\begin{align*}
		f^{(n+1)\star}(x) &\leq \sum_{i=1}^{n} \binom{n}{i} M_2^{n-i+1} h_i(x) f(x) +M_2^nf(x) + \sum_{i=1}^{n}\binom{n}{i}M_2^{n-i}h_{i+1}(x) f(x).
		\end{align*}
		Observe now that 
		$$
		\sum_{i=1}^{n} \binom{n}{i} M_2^{n-i+1} h_i(x) +M_2^n = \sum_{i=2}^{n} \binom{n}{i} M_2^{n-i+1} h_i(x) + (n+1) M_2^n h_1(x)
		$$
		and
		$$
		\sum_{i=1}^{n}\binom{n}{i}M_2^{n-i}h_{i+1}(x) = \sum_{i=2}^{n+1}\binom{n}{i-1}M_2^{n-i+1}h_{i}(x) 
		= \sum_{i=2}^{n}\binom{n}{i-1}M_2^{n-i+1}h_{i}(x) + h_{n+1}(x).
		$$
		Using these equalities and the standard identity for the binomial coefficients $\binom{n}{i} + \binom{n}{i-1} = \binom{n+1}{i}$, we finally get
		\begin{align*}
		f^{(n+1)\star}(x)
		&\leq \sum_{i=2}^{n} \binom{n+1}{i} M_2^{n-i+1} h_{i}(x) f(x) +  (n+1)M_2^{n} h_1(x) f(x) + h_{n+1}(x) f(x) \\
		&= \left(\sum_{i=1}^{n+1} \binom{n+1}{i} M_2^{n+1-i} h_i(x)\right) f(x),
		\end{align*}
		which is the claimed upper estimate. 
	This completes the proof of the theorem.
\end{proof}

The following corollary, giving the two-sided estimates for the densities $p_{\lambda}$, is a straightforward consequence of Theorem \ref{th:th1} and Lemma \ref{lem:lower_gen}.

\begin{corollary} \label{cor:poiss}
	Let \textup{\eqref{A}} hold. We have
	\begin{align*} 
		\lambda e^{- \lambda\left\|f\right\|_1  } f(x) \leq p_{\lambda}(x) \leq e^{(M_2-\left\|f\right\|_1) \lambda} \lambda f(x), \quad |x|<1, \ \lambda>0, 
	\end{align*}	
	and
	\begin{align*}
		1	 \leq \frac{p_{\lambda}(x)}{e^{-\lambda\left\|f\right\|_1 } f(x) \sum_{n = 1}^{\infty} \frac{\lambda^n h_n(x)}{n!}} \leq e^{M_2 \lambda}, \quad |x| \geq 1, \ \lambda>0,
	\end{align*}	
	where the constants $M_1, M_2$ come from Lemma \ref{lem:lower_gen} and Theorem \ref{th:th1}, respectively. Moreover, the functions $h_n$ can be replaced by the functions $g_n$ without changing constants.
\end{corollary}

\begin{proof}
	For the proof of the first estimate we assume that $|x| < 1$. The lower bound follows directly from the definition in \eqref{eq:compound_dens}, while the upper bound is a direct consequence of Lemma \ref{lem:es_for_dob} with $C = C_2$, where $C_2$ is the constant from (\ref{A}.c). Recall that $M_2 \geq C_2 \left\|f\right\|_1$.
	
Now consider $|x| \geq 1$. By the upper estimate of Theorem \ref{th:th1},
		\begin{align*}
			p_{\lambda}(x) = e^{-\lambda \norm{f}_1} \sum\limits_{n=1}^{\infty} \frac{\lambda^n f^{n \star}(x)}{n!} \leq e^{-\lambda \norm{f}_1} \sum\limits_{n=1}^{\infty} \sum\limits_{i=1}^{n} \frac{\lambda^n}{n!}\binom{n}{i} M_2^{n-i} h_i(x) f(x).
		\end{align*}
Thus, by Tonelli's theorem and the identity $\frac{1}{n!}\binom{n}{i} = \frac{1}{(n-i)!i!}$, we get
		\begin{align*}
			p_{\lambda}(x) \leq e^{- \lambda \norm{f}_1} \sum\limits_{i=1}^{\infty} \frac{\lambda^{i}h_i(x)}{i!}\sum_{n=i}^{\infty} \frac{(\lambda M_2)^{n-i}}{(n-i)!} f(x).
		\end{align*}
	Since  $\sum_{n=i}^{\infty} \frac{(\lambda M_2)^{n-i}}{(n-i)!} = e^{ M_2\lambda}$, we obtain the claimed upper estimate. The lower bound follows directly from the definition in \eqref{eq:compound_dens}, Lemma \ref {lem:lower_gen} and \eqref{eq:h-and-g}. 
\end{proof}

\subsection{Asymptotic properties of convolutions}

We first discuss the asymptotic properties of functions $g_n$ and $h_n$. 
	
	\begin{lemma} \label{lem:hn_as_prop} 
		Under assumption \textup{\eqref{A}} we have the following statements.
		\begin{itemize}
		  \item[(a)]  It holds that
				$$
			\sup_{x \in \R^d} g_2(x) < \infty \quad \Longleftrightarrow \quad \sup_{x \in \R^d} h_2(x) < \infty
			$$
			and
			$$
			g_2(x) \xrightarrow{|x| \to \infty} \infty \quad \Longleftrightarrow \quad h_2(x) \xrightarrow{|x| \to \infty} \infty.
			$$
			\item[(b)] If there exists a constant $C>0$ such that
			$$h_2(x) \leq C, \quad x \in \R^d,$$
			then, for every $n \geq 2$,
			$$h_n(x) \leq C^{n-1}, \quad x \in \R^d.$$
			 The same implication holds true for the functions $g_n$, $n \in \N$. 
			\item[(c)]  If 
			$$g_2(x) \xrightarrow{|x| \to \infty} \infty \quad \text{or} \quad h_2(x) \xrightarrow{|x| \to \infty} \infty,$$ 
			then, for every $n \geq 2$,  
			$$g_n(x) \xrightarrow{|x| \to \infty} \infty \quad \text{and} \quad h_n(x) \xrightarrow{|x| \to \infty} \infty.$$
			\item[(d)] If 
			$$g_2(x) \xrightarrow{|x| \to \infty} \infty \quad \text{or} \quad h_2(x) \xrightarrow{|x| \to \infty} \infty,$$ 
			then, for every $n > m \geq 1$,  
			$$\frac{h_m(x)}{h_n(x)} \xrightarrow{|x| \to \infty} 0.$$
			\item[(e)] If 
			$$g_2(x) \xrightarrow{|x| \to \infty} \infty \quad \text{or} \quad h_2(x) \xrightarrow{|x| \to \infty} \infty,$$ 
			then, for every $n \in \N$, 
			$$\frac{g_n(x)}{h_n(x)} \xrightarrow{|x| \to \infty} 1.$$ 
		\end{itemize}
	\end{lemma}
	
	\begin{proof} 
	We first show (a). Due to \eqref{eq:h-and-g} we only need to show two implications:
			$$
			\sup_{x \in \R^d} g_2(x) < \infty \quad \Longleftarrow \quad \sup_{x \in \R^d} h_2(x) < \infty
			$$
			and
			$$
			g_2(x) \xrightarrow{|x| \to \infty} \infty \quad \Longrightarrow \quad h_2(x) \xrightarrow{|x| \to \infty} \infty.
			$$
	These implications are direct consequence of \eqref{eq:lower_gen} and Theorem \ref{th:th1} applied with $n=2$. 
	
		We now use induction with respect to $n$. Recall that $h_n(x) = 0$ for $|x| \leq n$ and $n \geq 2$. 
		
	  For a proof of (b) it is enough to observe that if for some $n \geq 2$ we have $h_n(x) \leq C^{n-1}$, $x \in \R^d$, then
		$$
			h_{n+1}(x)= \frac{\int_{D(x)} f(x-y) f(y) h_n(y) dy}{f(x)} \leq \frac{C^{n-1}\int_{D(x)} f(x-y) f(y) dy}{f(x)} \leq C^{n}.
			$$
		 The proof of (b) for $g_n$'s is similar.
		
		The rest of the proof is based on the following observation. For every $n \geq 1$, $R>1$ and $|x| \geq R + 1$ we have
		\begin{align}
			\frac{\int_{D(x)\cap\{|y|\leq R\}} f(x-y) f(y) h_{n-1}(y) dy}{f(x)} \leq \frac{C_1^{\lceil R \rceil}f(x) \int_{\R^d} f(y) h_{n-1}(y) dy}{f(x)}\leq C_1^{\lceil R \rceil}\norm{f}^{n-1} < \infty. \label{eq:hn1}
			\end{align}
	Here the constant $C_1$ comes from (\ref{A}.b). Now, we observe that it follows from \eqref{eq:hn1} that 
			\begin{align}
	h_{n}(x)	\xrightarrow{|x| \to \infty} \infty \quad \text{implies} \quad	\frac{\int_{D(x)\cap\{|y| > R\}} f(x-y) f(y) h_{n-1}(y) dy}{f(x)} \xrightarrow{|x| \to \infty} \infty, \label{eq:hn2}
			\end{align}
		for every fixed $R > 1$. 
		
		We are now in a position to show (c). Suppose that $h_{2}(x)	\xrightarrow{|x| \to \infty} \infty$  or $g_{2}(x)	\xrightarrow{|x| \to \infty} \infty$. By (a) and  the induction hypothesis, also $h_{n}(x)	\xrightarrow{|x| \to \infty} \infty$, which means that there exists $R>1$ such that $h_n(y) \geq 1$ for $|y| > R$. In particular,
		$$
		h_{n+1}(x) = \frac{\int_{D(x)}f(x-y) f(y) h_n(y) dy}{f(x)} \geq \frac{\int_{D(x)\cap \{|y| > R\}} f(x-y) f(y) dy}{f(x)}.
		$$
	Applying \eqref{eq:hn2} for $n=2$ (recall that $h_1 = \I_{\R^d}$), we get that the term on the right hand side goes to $\infty$ as $|x| \to \infty$, which completes the proof of the induction step  for $h_{n+1}$. By \eqref{eq:h-and-g} the same is true for $g_{n+1}$. 
	
	To show (d), observe that 
	$$\frac{h_{m}(x)}{h_n(x)} =  \frac{h_{m}(x)}{h_{m+1}(x)} \cdot \ldots \cdot \frac{h_{n-1}(x)}{h_n(x)}, \quad m \geq 1, \ \ n > m +1.$$
	Therefore, it is enough to prove that $h_{2}(x)	\xrightarrow{|x| \to \infty} \infty$ implies
	$$
	\frac{h_{n-1}(x)}{h_n(x)} \xrightarrow{|x| \to \infty} 0, \quad n \geq 2.
	$$
	As for $n=2$ this holds trivially, we only have to check the induction step. By induction hypothesis, for every $\varepsilon >0$ there exists $R \geq 1$ such that $(2/\varepsilon) h_{n-1}(y) \leq h_n(y)$, for $|y| > R$. From part (b) we have that $h_n(x) \to \infty$ as $|x| \to \infty$. Using this and (\ref{eq:hn1}) -- (\ref{eq:hn2}), we get for sufficiently large $|x|$
\begin{align*}
			\frac{h_n(x)}{h_{n+1}(x)} & = \frac{\int_{D(x)\cap\{|y|\leq R\}} f(x-y) f(y) h_{n-1}(y) dy + \int_{D(x)\cap\{|y|> R\}} f(x-y) f(y) h_{n-1}(y) dy}{\int_{D(x)} f(x-y) f(y) h_{n}(y) dy} \\
			& \leq \frac{2 \int_{D(x)\cap\{|y|> R\}} f(x-y) f(y) h_{n-1}(y) dy}{(2/\varepsilon) \int_{D(x)\cap\{|y|> R\}} f(x-y) f(y) h_{n-1}(y) dy} \\
			                          & \leq \varepsilon.
\end{align*}
This completes the proof of (d).

It remains to show (e). By \eqref{eq:h-and-g}, \eqref{eq:lower_gen} and Theorem \ref{th:th1}, we may write
$$
1 \leq \frac{g_n(x)}{h_{n}(x)} \leq \sum_{i=1}^{n} {\binom{n}{i}} M_2^{n-i} \frac{h_i(x)}{h_n(x)}, \quad x \in \R^d, \ \ n \in \N.  
$$
By part (d), the sum on the right hand side goes to $1$ as $|x| \to \infty$, completing the proof of (e) and the entire lemma.  
\end{proof}
		
	The following corollary is a straightforward consequence of Theorem \ref{th:th1} and Lemma \ref{lem:hn_as_prop}.
	It shows that the functions $h_2$ and $g_2$ are decisive for the behaviour of $f^{n\star}$ and $p_{\lambda}$ at infinity.
	
	\begin{corollary} \label{cor:general_prop}
	Under \textup{\eqref{A}} we have the following statements. 
		\begin{itemize}
			\item[(a)] If there exists a constant $C>0$ such that $h_2(x) < C$, $x \in \R^d$, then
			$$
			n (M_1/2)^{n-1} f(x) \leq f^{n\star}(x) \leq n (M_2+C)^{n-1}  f(x), \quad |x| \geq 1, \ n \in \N,
			$$
			and
			$$
			e^{((M_1/2)-\left\|f\right\|_1) \lambda  } \lambda f(x) \leq p_{\lambda}(x) \leq e^{(M_2+C-\left\|f\right\|_1) \lambda  } \lambda f(x), \quad |x| \geq 1, \ \ \lambda>0,
			$$ 
			where the constants $M_1, M_2$ come from Lemma \ref{lem:lower_gen} and Theorem \ref{th:th1}, respectively.
			\item[(b)]  If 
			$$
			g_2(x) \xrightarrow{|x| \to \infty} \infty \quad \text{or, equivalently, } \quad h_2(x) \xrightarrow{|x| \to \infty} \infty,
			$$
			then, for every $n \in \N$,
			$$
			\frac{f^{n\star}(x)}{f(x)} = g_n(x)(1 + o(1)) = h_n(x)(1 + o(1)) \quad \text{as} \ \  |x| \to \infty.
			$$ 
			In particular, for all $1 \leq m < n$, 
			$$
			\frac{f^{n\star}(x)}{f^{m\star}(x)} \to \infty \quad \text{as} \ \  |x| \to \infty,
			$$
			and, for any $n \in \N$ and $\lambda>0$, 
				$$
			\frac{p_{\lambda}(x)}{f^{n\star}(x)} \to \infty \quad \text{as} \ \  |x| \to \infty.
			$$
			\end{itemize}
		
	\end{corollary}
		
	\begin{proof}
		\begin{itemize}
			\item[(a)] By the upper bound in Theorem \ref{th:th1} and Lemma \ref{lem:hn_as_prop} (b) we have
			\begin{align*}
				f^{n \star}(x) \leq &\left(\sum_{i=1}^n {\binom{n}{i}} M_2^{n-i} h_i(x) \right) f(x) \leq \left(\sum_{i=1}^n {\binom{n}{i}} M_2^{n-i} C^{i-1} \right) f(x).
			\end{align*}
			Because
			$$
			\sum_{i=1}^n {\binom{n}{i}} M_2^{n-i} C^{i-1} = \sum_{i=0}^{n-1} {\binom{n}{i+1}} M_2^{n-1-i} C^{i}
			$$
			and
			$$
			\binom{n}{i+1} \leq n \binom{n-1}{i},
			$$
			finally, by  Newton's binomial formula, we obtain
			$$
			f^{n \star}(x) \leq n (M_2 + C)^{n-1} f(x).
			$$
			The lower bound follows directly from Lemma \ref{lem:lower_gen}.
			
			The estimates for $p_{\lambda}$ follows easily from the the two-sided bound for $f^{n \star}$ obtained above and \eqref{eq:compound_dens}. 
			
			\item[(b)] Lemma \ref{lem:hn_as_prop} (d) implies 
			$$
			\sum_{i=1}^{n-1} {\binom{n}{i}} M_2^{n-i} h_i(x) = o(h_n(x)), \qquad \sum_{i=1}^{n-1} {\binom{n}{i}} M_1^{n-i} h_i(x) = o(h_n(x)).
			$$
			Therefore, it follows from \eqref{eq:h-and-g}, \eqref{eq:lower_gen} and Theorem \ref{th:th1} that
			$$
			f^{n\star}(x) = f(x) h_n(x)\big(1 + o(1)\big) = f(x) g_n(x)\big(1 + o(1)\big) \quad \text{as} \ |x| \to \infty.
			$$
			Moreover, for $1 \leq m < n$, 
			$$
			\frac{f^{n\star}(x)}{f^{m\star}(x)} = \frac{h_n(x)\big(1 + o(1)\big)}{h_m(x)\big(1 + o(1)\big)} = \frac{1 + o(1)}{\frac{h_m(x)}{h_n(x)}\big(1 + o(1)\big)}.
			$$
			Thus, by Lemma \ref{lem:hn_as_prop} (c)-(d) the ratio above tends to infinity as $|x| \to \infty$.
			For the proof of the last part we can take $m>n$ and observe that
			$$
			\frac{p_{\lambda}(x)}{f^{n\star}(x)} \geq e^{-\lambda \norm{f}_1} \frac{\lambda^m f^{m \star} (x)}{m! f^{n \star}(x)} \xrightarrow{x \to \infty} \infty,
			$$ 
			\end{itemize}
			for every $\lambda>0$. This completes the proof. 
	\end{proof}
	
\section{Estimates for exponential densities with doubling terms} \label{sec:Hn_Gn}
One of primary motivations for our study was to understand the asymptotic behaviour of convolutions for 
a class of multivariate densities described by the following assumption. 
	
	\medskip
	
		\begin{itemize}
\item[\bfseries(\namedlabel{B}{B})] 
    Let $f$ be a density of the form $f(x) := e^{-m|x|} g(|x|)$, where $m>0$ and $g:\R \to (0,\infty]$ is such that
    \begin{enumerate} 
    \item
    $g$ is decreasing function;
    \item
    there is a constant $C_3 \geq 1$ such that $g(r) \leq C_3 g(2r)$, for $r > 0$.
    \end{enumerate}
\end{itemize}

\medskip

\noindent
It is direct to check that \eqref{B} implies \eqref{A}. As explained in Introduction, for $d=1$ the formulas defining $g_n$'s and $h_n$ simplify, see \eqref{eq:one-d-h-and-g}. In multivariate case, the structure of integrals defining $g_n$'s and $h_n$'s is more complicated. In this section, we find estimates for these functions for $d \geq 2$.

\subsection{Upper estimates for functions $h_n$} \label{sec:upper_for_hn_Jn}
 Let
		$$
		H_1 \equiv \I_{[0,\infty)}
		$$
		and define inductively: $H_{n+1} \equiv 0$ on $[0,2]$ and
		$$
		H_{n+1}(r) :=  \frac{1}{g(r)r^{\frac{d-1}{2}}} \int_{1}^{r-1} g(r-\rho) (r-\rho)^{\frac{d-1}{2}} g(\rho)  \rho^{\frac{d-1}{2}}  H_{n}(\rho)  d\rho, \quad r > 2, \ \ n \in \N. 
		$$	
    It follows directly from the definition that $\supp H_n \subset (n,\infty)$, for all $n \in \N$.  Clearly, for $d=1$ we have $H_n(|x|) = h_n(x)$, $n \in \N$, $x \in \R$.

Our next result gives the upper bound for the functions $h_n$ in terms of functions $H_n$ for $d \geq 2$.
	
		\begin{theorem} \label{th:th2}
		 Let $d \geq 2$.  Under assumption \eqref{B} we have
		$$
		h_n(x) \leq M_3^{n-1} H_n(|x|), \quad x \in \R^d, \ n \geq 1,
		$$ 
		where
		$$
		M_3 := 1 \vee \left(\frac{2 \pi^{\frac{d-1}{2}}}{\Gamma \left(\frac{d-1}{2}\right)} \left(\int_0^{\infty}  \exp{\left(-\frac{m}{\pi^2}\frac{s^2}{\sqrt{1+s^2}}\right)} s^{d-2} \, ds \right)\right)
		$$
	\end{theorem} 
	
\noindent
Observe that the constant $M_3$ does not depend on the function $g$.

The proof of Theorem \ref{th:th2} requires an auxiliary lemma.
It allows us to estimate the exponential functions that appear under the integrals defining $g_n$ and $h_n$.
	
\begin{lemma}\label{le:le1}
	Let $d \geq 2$. For $x = (x_1,0,\ldots,0), \; y=(y_1,\ldots,y_d) \in \R^d$ such that 
	$x \neq y$ and $y \neq 0$, we have
	$$
	|x-y| + |y| = |x| + \frac{(y_2^2+\cdots+y_d^2)}{|y|+y_1}+ \frac{(y_2^2+\cdots+y_d^2)}{|x-y|+(x_1-y_1)}.
	$$
\end{lemma}

\begin{proof}
	Let $r^2 = (y_2^2+\cdots+y_d^2)$. We have
\begin{align*}
\frac{r^2}{|x-y|+(x_1-y_1)} & = \frac{r^2(|x-y| - (x_1-y_1))}{(|x-y| + (x_1-y_1))(|x-y| - (x_1-y_1))}  \\ & = \frac{r^2(|x-y| - (x_1-y_1))}{(x_1-y_1)^2+r^2-(x_1-y_1)^2} = |x-y|-x_1 + y_1
\end{align*}
and
$$		
\frac{r^2}{|y|+y_1} = |y| - y_1 .
$$
By summing up on both sides of the above equalities, we get the assertion. 
\end{proof}

	\begin{proof}[Proof of Theorem \ref{th:th2}]
		We use induction. We have $h_1(x) = H_1(|x|) =  1$, $x \in \R^d$, so the claimed inequality is true for $n=1$. Observe that for $n \geq 2$ and $|x| \leq n$ one has $h_n(x) = 0$ and the assertion holds trivially. We are left to consider $|x| > n$. 
		In order to check the induction step, we have to introduce some more notation
		$$
		D(x)_{-} = D(x) \cap \left\{z\in\R^d: |z| < |x-z|\right\}, \quad
		D(x)_{+} = D(x) \cap \left\{z\in\R^d: |z| \geq |x-z|\right\}.
		$$
		Since every $h_n$ is radial, we may and do assume that $x=(x_1,0,\ldots,0)$ with $x_1>n \geq 2$.
		
		It follows from Lemma \ref{le:le1} that 
		\begin{equation*}\label{eq:expeq1}
			\frac{\exp(-m|x-y|)\exp(-m|y|)}{\exp(-m|x|)} \leq \exp\left(\frac{-m(y_2^2+\cdots+y_d^2)}{|y|+y_1}\right), \quad y \in D(x)_{-},
		\end{equation*}
		and 
		\begin{equation*}\label{eq:expeq2}
			\frac{\exp(-m|x-y|)\exp(-m|y|)}{\exp(-m|x|)}  \leq \exp \left(\frac{-m(y_2^2+\cdots+y_d^2)}{|x-y|+(x_1-y_1)}\right), \quad y \in D(x)_{+}.
		\end{equation*}
		Hence, by induction hypothesis,  
		\begin{align*}
			h_n(x) & = \int_{D(x)} \frac{f(x-y)f(y)}{f(x)}h_{n-1}(y)dy \nonumber \\
			& \leq \frac{M_3^{n-2}}{g(|x|)}\int_{D(x)}\frac{\exp{(-m|x-y|)}\exp{(-m|y|)}}{\exp{(-m|x|)}}g(|x-y|)g(|y|)H_{n-1}(|y|)dy \\
			& \leq \frac{M_3^{n-2}}{g(|x|)}\underbrace{\int_{D(x)_{-}}\exp\left(\frac{-m(y_2^2+\cdots+y_d^2)}{|y|+y_1}\right)g(|x-y|)g(|y|)H_{n-1}(|y|)dy}_{=:\In_1} \\
			& \ \ \ \	+ \frac{M_3^{n-2}}{g(|x|)}\underbrace{\int_{D(x)_{+}} \exp \left(\frac{-m(y_2^2+\cdots+y_d^2)}{|x-y|+(x_1-y_1)}\right) g(|x-y|)g(|y|)H_{n-1}(|y|)dy}_{=:\In_2}.
		\end{align*}	
		
		Recall that $x=(x_1,0,\ldots,0)$ with $x_1>n \geq 2$ (in particular, $|x|=x_1$). When $d \geq 3$, then we may simplify our further calculations by reducing the integration to the
		subset of the plane -- we introduce the hyper-spherical coordinates for $(y_2,\ldots,y_d)$ in $\R^{d-1}$. By this, 
		\begin{align*}
			\In_1 = \frac{2\pi^{\frac{d-1}{2}}}{\Gamma \left(\frac{d-1}{2}\right)} \int_{1}^{\frac{x_1}{2}}dy_1 & \int_{0}^{\sqrt{(x_1-1)^2-(x_1 - y_1)^2}} \exp{\left(\frac{-mr^2}{\sqrt{y_1^2+r^2}+y_1}\right)} \times \nonumber \\
		& \times g\big(\sqrt{r^2+y_1^2}\big)g\big(\sqrt{(x_1-y_1)^2+r^2}\big) H_{n-1}\big(\sqrt{r^2+y_1^2}\big) \, r^{d-2} \, dr 
		\end{align*}
		and
		\begin{align*}
		 \In_2 = \frac{2\pi^{\frac{d-1}{2}}}{\Gamma \left(\frac{d-1}{2}\right)} \int_{\frac{x_1}{2}}^{x_1-1}dy_1 & \int_{0}  ^{\sqrt{(x_1-1)^2-y_1^2}}  
			 \exp{\left(\frac{-mr^2}{\sqrt{(x_1-y_1)^2+r^2}+x_1-y_1}\right)} \times \nonumber \\
			& \times g\big(\sqrt{r^2+y_1^2}\big)g\big(\sqrt{(x_1-y_1)^2+r^2}\big) H_{n-1}\big(\sqrt{r^2+y_1^2}\big) \, r^{d-2} \, dr .
		\end{align*}
		For $d=2$ the integrals $\In_1$ and $\In_2$ take this form directly. This is because of the symmetry of the domains
$D(x)_-$, $D(x)_+$ about the line $y_2 = 0$. The domains of integration under $\In_1$ and $\In_2$ are subsets of the sets $\big\{(y_1,r) \in \R^2: 1 \leq \sqrt{y_1^2 + r^2} \leq x_1/2, \, y_1 \geq 0, r \geq 0 \big\}$ and $\big\{(y_1,r) \in \R^2: x_1/2 \leq \sqrt{y_1^2 + r^2} \leq x_1 -1, \, y_1 \geq 0, r \geq 0 \big\}$, respectively. Therefore, by using polar coordinates $y_1 = \rho\cos\alpha$, $r=\rho\sin\alpha$, we get
\begin{align*}
			\In_1 \leq \frac{2\pi^{\frac{d-1}{2}}}{\Gamma \left(\frac{d-1}{2}\right)}  & \int_{1}^{\frac{x_1}{2}}d\rho \int_0^{\frac{\pi}{2}}  \exp{\left(\frac{-m\rho^2\sin^2\alpha}{\rho + \rho \cos \alpha}\right)} \times \\
			& \times g(\rho)g\big(\sqrt{(x_1-\rho\cos\alpha)^2+\rho^2\sin^2\alpha}\big) H_{n-1}(\rho)\, \rho^{d-1} \, \sin^{d-2}\alpha \, d\alpha
\end{align*}
and
\begin{align*}			
			\In_2 \leq \frac{2\pi^{\frac{d-1}{2}}}{\Gamma \left(\frac{d-1}{2}\right)}  & \int_{\frac{x_1}{2}}^{x_1-1}d\rho \int_0^{\frac{\pi}{2}}  \exp{\left(\frac{-m\rho^2\sin^2\alpha}{\sqrt{(x_1-\rho\cos\alpha)^2+\rho^2\sin^2\alpha}+x_1-\rho\cos\alpha}\right)} \times  \\
			& \times g(\rho)g\big(\sqrt{(x_1-\rho\cos\alpha)^2+\rho^2\sin^2\alpha}\big) H_{n-1}(\rho)\, \rho^{d-1} \, \sin^{d-2}\alpha \, d\alpha .
\end{align*}
	Now, observe that
		\begin{align}  \label{eq:cos_simplify}
		(x_1-\rho\cos\alpha)^2+\rho^2\sin^2\alpha = x_1^2-2\rho x_1 \cos\alpha+\rho^2-2\rho x_1+2\rho x_1 =
		(x_1-\rho)^2+2\rho x_1 (1-\cos\alpha).
		\end{align}
		Together with the standard inequality 
		\begin{align} \label{eq:taylor_cos}
		1-\cos \alpha \leq \frac{\alpha^2}{2},
		\end{align}
		it implies that
		\begin{align*}
		\sqrt{(x_1-\rho\cos\alpha)^2+\rho^2\sin^2\alpha}+x_1-\rho\cos\alpha & \leq 2 \sqrt{(x_1-\rho\cos\alpha)^2+\rho^2\sin^2\alpha} \\
		& \leq 2  \sqrt{(x_1-\rho)^2+ \rho x_1\alpha^2 }.
		\end{align*}
		Moreover, 
		$$
		\sqrt{(x_1-\rho\cos\alpha)^2+\rho^2\sin^2\alpha} \geq x_1-\rho >0 
		$$
		and 
		$$
		\frac{2}{\pi} \alpha \leq \sin \alpha \leq \alpha,
		$$
		for every $\alpha \in [0,\pi/2]$. Therefore, by using all these three estimates and the monotonicity of the exponential function and the profile $g$, we can further estimate
	 \begin{align*}
	 	\In_1 \leq \int_{1}^{\frac{x_1}{2}}g(\rho)g(x_1-\rho) H_{n-1}(\rho)\, \rho^{d-1} \left(\int_0^{\frac{\pi}{2}}  \exp\left(-\frac{2m}{\pi^2}\rho \alpha^2\right) 
	 	\, \alpha^{d-2} \, d\alpha\right) d\rho,
	\end{align*}
	\begin{align*}
	 	\In_2 \leq \int_{\frac{x_1}{2}}^{x_1-1}g(\rho)g(x_1-\rho) H_{n-1}(\rho)\, \rho^{d-1} \left(\int_0^{\frac{\pi}{2}}  \exp{\left(-\frac{2m}{\pi^2}\frac{\rho^2 \alpha^2}{\sqrt{(x_1-\rho)^2+ \rho x_1\alpha^2}}\right)} 
	 	\, \alpha^{d-2} \, d\alpha\right) d\rho.
	 \end{align*}
	 We will now consider only the inner integrals and change variables according to $\alpha= \sqrt{\frac{x_1-\rho}{\rho x_1}}s$. 
	For $ 1 \leq \rho \leq x_1/2$ we have
	 \begin{equation}
	 	\rho\alpha^2 = \rho\frac{x_1-\rho}{\rho x_1}s^2 \geq \frac{s^2}{2}
	 \end{equation}
	and the inner integral in $\In_1$ can be estimated from above as follows
	\begin{equation}
	\left(\frac{x_1-\rho}{\rho x_1}\right)^{\frac{d-1}{2}} \int_0^{\infty}  \exp{\left(-\frac{m}{\pi^2}s^2\right)} 
	\, s^{d-2} \, ds.
	\end{equation}
	On the other hand, for $ x_1/2 \leq \rho \leq x_1-1$ we have
	 $$
	 \frac{\rho^2 \alpha^2}{\sqrt{(x_1-\rho)^2+ \rho x_1\alpha^2}} = \frac{\frac{\rho^2}{\rho x_1} (x_1-\rho) s^2}{\sqrt{(x_1-\rho)^2+ (x_1-\rho) s^2 }}
	 \geq \frac{\frac{\rho}{x_1} (x_1-\rho) s^2}{\sqrt{(x_1-\rho)^2 (1+ s^2) }}   
	 \geq \frac{s^2}{2\sqrt{1+ s^2 }}.
	 $$
	 This easily gives that the inner integral in $\In_2$ is bounded above by 
	 $$
	 \left(\frac{x_1-\rho}{\rho x_1}\right)^{\frac{d-1}{2}} \int_0^{\infty}  \exp{\left(-\frac{m}{\pi^2}\frac{s^2}{\sqrt{1+s^2}}\right)} 
	 \, s^{d-2} \, ds.
	 $$
	 By inserting these bounds into $\In_1$, $\In_2$, we conclude that
	 \begin{align*}
	 \In_1 \leq \frac{2 \pi^{\frac{d-1}{2}}}{\Gamma \left(\frac{d-1}{2}\right)x_1^{\frac{d-1}{2}}}  \left(\int_0^{\infty}  \exp{\left(-\frac{m}{\pi^2}s^2\right)}
	 \, s^{d-2} \, ds \right) \int_{1}^{\frac{x_1}{2}}g(x_1-\rho)(x_1-\rho)^{\frac{d-1}{2}}g(\rho)\rho^{\frac{d-1}{2}} H_{n-1}(\rho) d \rho
	\end{align*}
	and
	\begin{align*}
	 \In_2 \leq \frac{2 \pi^{\frac{d-1}{2}}}{\Gamma \left(\frac{d-1}{2}\right)x_1^{\frac{d-1}{2}}}  \left(\int_0^{\infty}  \exp{\left(-\frac{m}{\pi^2}\frac{s^2}{\sqrt{1+s^2}}\right)} 
	 \, s^{d-2} \, ds \right)\int_{\frac{x_1}{2}}^{x_1-1}g(x_1-\rho)(x_1-\rho)^{\frac{d-1}{2}}g(\rho) \rho^{\frac{d-1}{2}} H_{n-1}(\rho) d \rho.
	 \end{align*}	
Since $s^2/\sqrt{1+s^2} \leq s^2$, $s \geq 0$, we finally get the claimed bound 
 	\begin{align*}
 	h_n(x)  \leq	\frac{M_3^{n-1}}{g(|x|)|x|^{\frac{d-1}{2}}} \int_{1}^{|x|-1}g(|x|-\rho)g(\rho)(|x|-\rho)^{\frac{d-1}{2}}\rho^{\frac{d-1}{2}} H_{n-1}(\rho) d \rho. 
 	\end{align*}
	with 
	$$
	M_3 = 1 \vee \frac{2 \pi^{\frac{d-1}{2}}M_3^{n-2}}{\Gamma \left(\frac{d-1}{2}\right)}  \left(\int_0^{\infty}  \exp{\left(-\frac{m}{\pi^2}\frac{s^2}{\sqrt{1+s^2}}\right)} 
 		\, s^{d-2} \, ds \right).
		$$
	\end{proof}

	\subsection{Lower estimates for functions $g_n$} \label{sec:lower_for_gn_Jn}

	The lower bound for the functions $g_n$ will be given in terms of different functions $G_{n}$.
	Define:
	$$
	G_1 \equiv \I_{[0,\infty)}
	$$
	and
	$$
	G_{n+1}(r) := \begin{cases}
          \frac{1}{g(r)} \int_{0}^{r} g(r-\rho)g(\rho)  G_{n}(\rho)  d\rho & \text{if\ \ } d=1,\vspace{0.2cm}\\ 
         \frac{1}{g(r)r^{\frac{d-1}{2}}}  \int_{0}^{r} \left(\int_0^{\sqrt{\rho \wedge (r-\rho)}}  e^{-m s^2}  s^{d-2}\, ds\right) g(r-\rho)  (r-\rho)^{\frac{d-1}{2}} g(\rho)\rho^{\frac{d-1}{2}}  G_{n}(\rho)  d\rho & \text{if\ \ } d \geq 2,
    \end{cases}
	$$
	for all $n \in \N$ and $r > 0$.  For convenience, we also put $G_n(0)=0$, $n \geq 2$.  It is then clear that for $d=1$, by definition, we have $G_n(|x|) = g_n(x)$, $n \in \N$, $x \in \R$. 
	
	\begin{theorem} \label{th:th4}
		 Let $d \geq 2$. Under assumption {\bf \eqref{B}} we have
		$$
		g_n(x) \geq M_4^{n-1} G_n(|x|), \quad x \in \R^d, \ n \geq 1,
		$$
		with
		$$
		M_4 :=1 \wedge 
		\left(\frac{2^{d-1}}{C_3 \pi^{\frac{d-3}{2}}\Gamma \left(\frac{d-1}{2}\right)}\right).
		$$
	\end{theorem}
		
	\begin{proof}
		As before, we only need to check the induction step.  
		Let $x=(x_1,0,\ldots,0)$ with $x_1 > 0$. By induction hypothesis, 
		\begin{align} \label{eq:lower_one}
			g_n(x) & = \int_{E(x)} \frac{f(x-y)f(y)}{f(x)}g_{n-1}(y)dy  \nonumber \\
			& \geq M_4^{n-2}\int_{E(x)} \frac{f(x-y)f(y)}{f(x)}G_{n-1}(|y|)dy.
		\end{align}	
		By Lemma \ref{le:le1} the expression on the right hand side of  \eqref{eq:lower_one} is equal to
			\begin{align*}
			 \frac{M_4^{n-2}}{g(|x|)} \int_{E(x)} \exp{\left(\frac{-m(y_2^2+\cdots+y_d^2)}{|y|+y_1}\right)} & \exp{\left(\frac{-m(y_2^2+\cdots+y_d^2)}{|x-y|+(x_1-y_1)}\right)} g(|x-y|)g(|y|) G_{n-1}(|y|)dy.
		\end{align*}	
		Now, similarly as above, using the hyper-spherical coordinates for $(y_2,\cdots,y_d)$ in $\R^{d-1}$ for $d \geq 3$, and the symmetry of the domain $E(x)$ about the line $y_2 = 0$ for $d=2$, we obtain that the above expression is equal to
		\begin{align} \label{eq:th4_aux_0}
			\frac{2\pi^{\frac{d-1}{2}}M_4^{n-2}}{\Gamma \left(\frac{d-1}{2}\right)} \frac{1}{g(x_1)}\int_{\widetilde E} & \exp{\left(\frac{-mr^2}{\sqrt{r^2+y_1^2}+y_1}\right)}  \exp{\left(\frac{-mr^2}{\sqrt{(x_1-y_1)^2+r^2}+x_1-y_1}\right)} \times \nonumber\\
			& \times g\big(\sqrt{(r^2+y_1^2)}\big)g\big(\sqrt{(x_1-y_1)^2+r^2}\big) G_{n-1}\big(\sqrt{(r^2+y_1^2)}\big)r^{d-2}dr dy_1,
		\end{align}
		where $\widetilde E:= E((x_1,0)) \cap \left\{(y_1,r):r>0 \right\}$. 	This domain can be reduced to the set $\big\{(y_1,r) \in \R^2 : 0 \leq \sqrt{y_1^2+r^2} \leq x_1, \, 0 \leq r \leq \sqrt{3} y_1  \big\}$ (see Figure \ref{fig3} for illustration)
		\begin{figure}\centering
			\includegraphics[width = \textwidth]{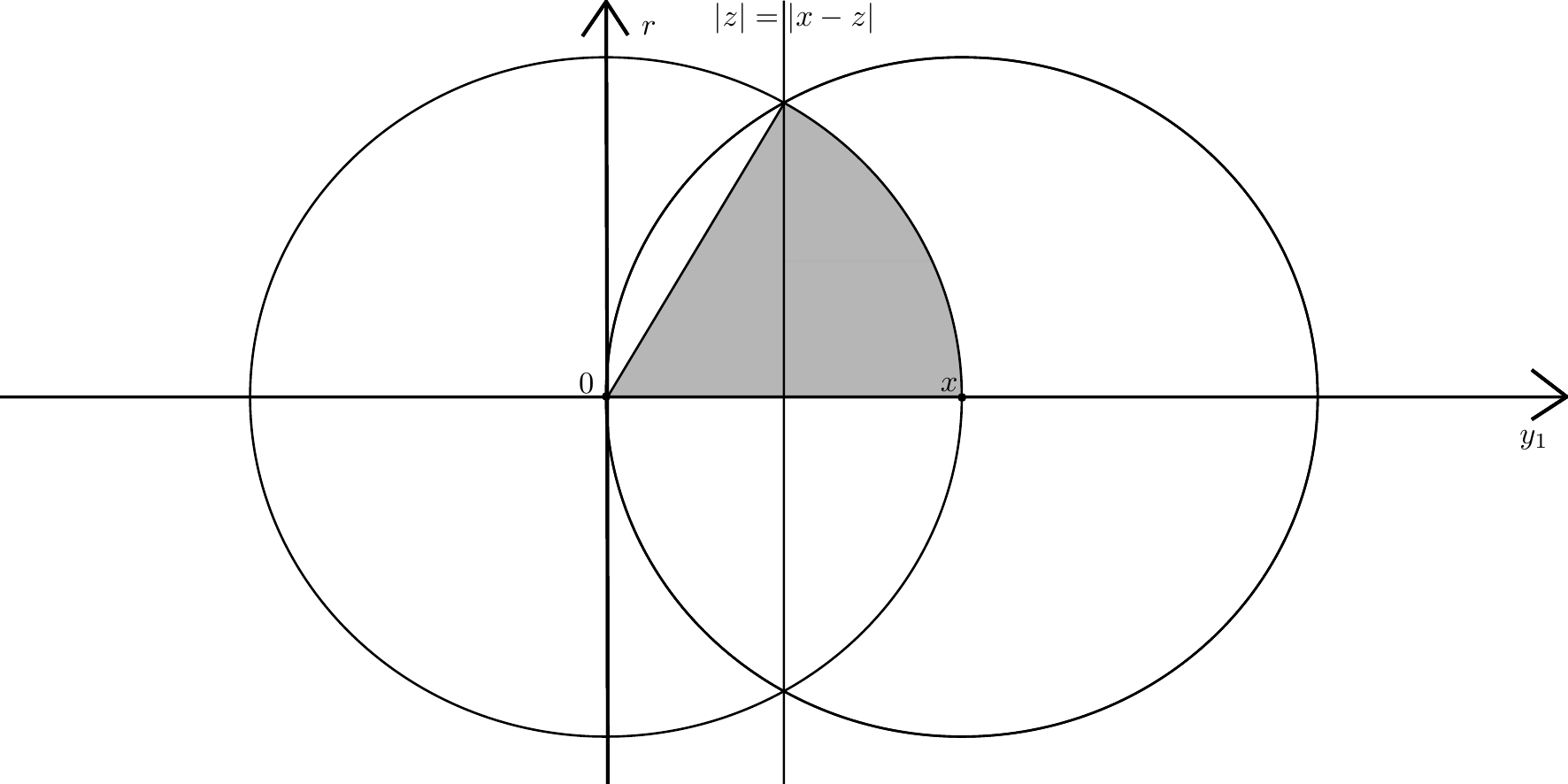}
			\caption{Illustration to the reduction of the domain of integration in the proof of Theorem \ref{th:th4}.}\label{fig3}
		\end{figure}
		which becomes a rectangle in polar coordinates $r=\rho\sin\alpha, y_1 = \rho\cos\alpha$. This leads to the following lower bound for the double integral in \eqref{eq:th4_aux_0}
		\begin{align*}
			 \int_{0}^{x_1}d\rho \int_0^{\frac{\pi}{3}} & \exp{\left(\frac{-m\rho^2\sin^2\alpha}{\rho+\rho\cos\alpha}\right)}  \exp{\left(  \frac{-m\rho^2\sin^2\alpha}{\sqrt{(x_1-\rho\cos\alpha)^2+\rho^2\sin^2\alpha}+x_1-\rho\cos\alpha}\right)} \times \\
			& \ \times g(\rho)g(\sqrt{(x_1-\rho\cos\alpha)^2+\rho^2\sin^2\alpha}) G_{n-1}(\rho)\rho^{d-1}\sin^{d-2}\alpha \, d\alpha.
		\end{align*}
Next, by the monotonicity of the exponential function and the profile $g$, \eqref{eq:cos_simplify}--\eqref{eq:taylor_cos}, and the standard inequalities $\frac{2}{\pi} \alpha \leq \sin \alpha \leq \alpha$, $\alpha \in [0,\pi/3]$, we can estimate the above expression from below by 
		\begin{align} \label{eq:th4_aux}
			 \left(\frac{2}{\pi}\right)^{d-2}  \int_{0}^{x_1} & \left(\int_0^{\frac{\pi}{3}} \exp{(-m\rho \alpha^2)}  \exp{\left(  \frac{-m\rho^2\alpha^2}{x_1-\rho}\right)} g(\sqrt{(x_1-\rho)^2+\rho x_1 \alpha^2}) \alpha^{d-2} \, d\alpha\right) \times \\ 
			& \times g(\rho) G_{n-1}(\rho)\rho^{d-1} d\rho. \nonumber
		\end{align}
		Changing variables according to $\alpha= \sqrt{\frac{x_1-\rho}{\rho x_1}}s$ and using the identity
		$$
		\rho \alpha^2 + \frac{\rho^2\alpha^2}{x_1-\rho} = \frac{\rho x_1 \alpha^2}{x_1-\rho} = s^2
		$$
		 and the estimate
		$$
		\sqrt{\frac{\rho x_1}{x_1-\rho}} \geq \sqrt{\rho} \geq \sqrt{\rho \wedge (x_1-\rho)},
		$$
		we get that the inner integral in \eqref{eq:th4_aux} is greater than or equal to
		\begin{align*}
			\left(\frac{x_1-\rho}{\rho x_1}\right)^{\frac{d-1}{2}} & \int_0^{{\frac{\pi}{3}\sqrt{\rho \wedge (x_1-\rho)}}} e^{-m s^2} g(\sqrt{(x_1-\rho)^2+(x_1-\rho)  s^2})s^{d-2}\, ds\\
			& \geq \left(\frac{x_1-\rho}{\rho x_1}\right)^{\frac{d-1}{2}}\int_0^{{\sqrt{\rho \wedge (x_1-\rho)}}} e^{-m s^2} g(2(x_1-\rho)) s^{d-2}\, ds.
		\end{align*} 
		Finally, by using the doubling property $C_3 g(2(x_1-\rho)) \geq g(x_1-\rho)$ as in (\ref{B}.b), we get the estimate
		\begin{align*}
		g_n(x) \geq \frac{2^{d-1} M_4^{n-2}}{C_3 \pi^{\frac{d-3}{2}}\Gamma \left(\frac{d-1}{2}\right)}  \frac{1}{g(|x|)|x|^\frac{d-1}{2}}  &
		\int_0^{|x|}\left(\int_0^{\sqrt{\rho \wedge (|x|-\rho)}}  e^{-m s^2}  s^{d-2}\, ds\right) \times 
		\\ & \times g(|x|-\rho)  (|x|-\rho)^{\frac{d-1}{2}} g(\rho)\rho^{\frac{d-1}{2}}  G_{n}(\rho)  d\rho.
		\end{align*}
		This gives the claimed bound with the constant
		$
		M_4 = 
		\frac{2^{d-1}}{C_3 \pi^{\frac{d-3}{2}}\Gamma \left(\frac{d-1}{2}\right)}
		$
		completing the proof.
\end{proof}

	\section{Applications} \label{sec:appl}

	We now apply the estimates obtained in the previous chapter to give two-sided bounds for densities $f$ as in \eqref{eq:exp_dens} and \eqref{eq:exp_dens_cut} with  $\gamma \in \big[0,\frac{d+1}{2}\big)$. 
	The borderline case $\gamma =\frac{d+1}{2}$ requires different approach -- it is partly resolved in our forthcoming paper \cite{BK22}. 

\subsection{Special functions}

Our estimates in this chapter are based on a use of some classical special functions. For reader's convenience we first recall the definitions and identities which will be needed below.

\begin{itemize}
\item
Beta function \cite[5.12.1]{NIST:DLMF}: 
\begin{align}\label{eq:Beta_def}
B(a,b):= \int_0^1 t^{a-1}(1-t)^{b-1} dt = \frac{\Gamma(a)\Gamma(b)}{\Gamma(a+b)}, \quad a,b >0;
\end{align}
\item
Incomplete beta function \cite[8.17.1]{NIST:DLMF}:
$$
B_x(a,b):=\int_0^x t^{a-1}(1-t)^{b-1} dt, \quad a, b >0, \ \ x \in [0,1];
$$
\item (Gauss) hypergeometric function \cite[15.2.1]{NIST:DLMF}:
$$
F(a,b,c;x)= \frac{\Gamma(c)}{\Gamma(a)\Gamma(b)}\sum_{s=0}^{\infty} \frac{\Gamma(a+s)\Gamma(b+s)}{\Gamma(c+s)s!}x^s, \quad a,b,c>0, \ \ |x| < 1;
$$
Below we use the following identity \cite[8.17.8]{NIST:DLMF}
\begin{align}\label{eq:iBeta_hypergeom}
	B_x(a,b) = \frac{x^a(1-x)^b}{a}F(a+b,1,a+1;x).
\end{align}
\item Generalized Bessel function (Wright function) \cite{Wright},\cite[10.46.1]{NIST:DLMF}:
\begin{align} \label{eq:GBf}
\phi(\rho, \beta; t) := \sum_{n=0}^{\infty} \frac{t^n}{\Gamma(\rho n + \beta) n!}, \quad \rho>0, \ \ \beta \geq 0, \ \ t > 0.
\end{align}
We need the following asymptotic estimates of the function $\phi(\rho,\beta; t)$ as $t \to \infty$. 

\begin{lemma}{\cite[Theorem 2]{Wright}}\label{lem:wright} We have 
	$$
	\phi(\rho,\beta;t) = (\rho t)^{\frac{1-2\beta}{2\rho +2}}\exp\left(\bigg(1+\frac{1}{\rho}\bigg)(\rho t)^{\frac{1}{\rho + 1}}\right) \left( \frac{1}{\sqrt{2\pi(\rho+1)}} +  O\left(\frac{1}{ t^{\frac{1}{1+\rho}}}\right) \right), \quad \text{as} \ \ t \to \infty.
	$$
	In particular, there are constants $D_1,D_2 >0$ (depending on $\rho$ and $\beta$) such that
	$$
	D_1 \leq \frac{\phi(\rho,\beta;t)}{t^{\frac{1-2\beta}{2\rho +2}}\exp\left(\big(1+1/\rho\big)(\rho t)^{\frac{1}{\rho + 1}}\right)} \leq D_2 , \quad t \geq 1.
	$$
\end{lemma}
\end{itemize}

\subsection{Direct estimates for dimension one} \label{sec:direct_1d} We first analyse densities \eqref{eq:exp_dens} for $d=1$. In this case our estimates are sharpest.  

\begin{lemma}\label{lem:direct_1d_gn}
	Let $d=1$ and $f$ be as in \eqref{B} with $g(r) = r^{-\gamma}$ where $\gamma \in \big[0,1)$. Then
$$
g_n(x) =\frac{\Gamma(1-\gamma)^n}{\Gamma((1-\gamma)n)} |x|^{(1-\gamma)(n-1)}, \quad x \in \R, \  n \in \N.
$$
\end{lemma}
\begin{proof}
For $n=1$ this is clear because $g_1 \equiv \I_{\R}$. We only need to consider $x >0$. By induction, direct substitution and \eqref{eq:Beta_def},
\begin{align*}
g_{n+1}(x) & = \int_0^x \frac{f(x-s)}{f(x)} f(s) g_n(s) ds  = \frac{\Gamma(1-\gamma)^n}{\Gamma((1-\gamma)n)} x^{\gamma} \int_0^x (x-s)^{-\gamma} s^{(1-\gamma)(n-1)-\gamma} ds \\
      &= \frac{\Gamma(1-\gamma)^n}{\Gamma((1-\gamma)n)} \frac{\Gamma(1-\gamma)\Gamma((1-\gamma)n)}{\Gamma((1-\gamma)(n+1))} x^{(1-\gamma)n} = \frac{\Gamma(1-\gamma)^{n+1}}{\Gamma((1-\gamma)(n+1))} x^{(1-\gamma)n}.
\end{align*}
\end{proof}

\begin{theorem}\label{th:1d}
	Let $d=1$ and let $f$ be as in \eqref{B} with $g(r) = r^{-\gamma}$ where $\gamma \in \big[0,1)$. 
	Then we have the following estimates.
	\begin{itemize}
	\item[(1)] For $|x| \geq 1$ and $n \in \N$, 
	$$
  \frac{\Gamma(1-\gamma)^n}{\Gamma((1-\gamma)n)} \leq \frac{f^{n\star}(x)}{f(x) |x|^{(1-\gamma)(n-1)}} \leq \frac{\Gamma(1-\gamma)^n}{\Gamma((1-\gamma)n)} +  \frac{2M_2 n M_3 (M_2+\Gamma(1-\gamma))^{n-1} }{|x|^{1-\gamma}}.  
	$$
	In particular,
	$$
	\lim_{|x| \to \infty} \frac{f^{n\star}(x)}{f(x) |x|^{(1-\gamma)(n-1)}} = \lim_{|x| \to \infty} \frac{f^{n\star}(x)}{e^{-m|x|} |x|^{(1-\gamma)n-1}} = \frac{\Gamma(1-\gamma)^n}{\Gamma((1-\gamma)n)}.
	$$
	\item[(2)] For $|x| \geq 1$ and $\lambda>0$,
	$$
	1 \leq \frac{p_{\lambda}(x)}{ e^{- \lambda\norm{f}_1} e^{-m|x|} |x|^{-1} \phi\big(1-\gamma,0;\Gamma(1-\gamma) \lambda |x|^{1-\gamma}\big) } \leq e^{M_2\lambda}, 
	$$
	where $\phi$ is the generalized Bessel function defined in \eqref{eq:GBf}. In particular, there are positive constants $E_1, E_2, E_3$ and $E_4$ (depending on $\gamma$) for which the following estimates holds: \\
if $\lambda |x|^{1-\gamma} \leq 1$, then 
  $$
			E_1 \leq \frac{p_{\lambda}(x)}{ \lambda e^{- \lambda\norm{f}_1 } e^{-m|x|} |x|^{-\gamma}}  \leq E_2;
  $$
if $\lambda |x|^{1-\gamma} \geq 1$, then
$$ 
	E_3 \leq \frac{p_{\lambda}(x) }{\lambda^{\frac{1}{4-2\gamma}}|x|^{ -\frac{3-\gamma}{4-2\gamma}}\exp\left(- \lambda\norm{f}_1 -m|x|+\frac{2-\gamma}{1-\gamma} \big(\Gamma(2-\gamma)\lambda |x|^{1-\gamma}\big)^{\frac{1}{2 - \gamma}}\right)} \leq  E_4 e^{M_2\lambda} .
$$
\end{itemize}
\end{theorem}

\begin{proof}
We first show (1). The lower bound follows from Lemmas \ref{lem:lower_gen} and \ref{lem:direct_1d_gn}.

 We now establish the upper bound. We need the following estimate for the gamma function. Since $(0,\infty) \ni t \mapsto \Gamma(t)$ is a convex function  and $\Gamma(1)=\Gamma(2)=1$, we have that $\Gamma(t) \geq 1$ for $t \in (0,1]\cup[2,\infty)$. Furthermore, $\Gamma(t) = \frac{\Gamma(t+1)}{t} > \frac{1}{2}$, for all $t \in (1,2)$, which implies 
\begin{align} \label{eq:gamma_est}
\Gamma(t) \geq \frac{1}{2}, \quad t>0.
\end{align}
By Theorem \ref{th:th1}, the inequality $\binom{n}{i} \leq n \binom{n - 1}{i}$,  \eqref{eq:gamma_est} and $|x| \geq 1$, we have
\begin{align*}
	\frac{f^{n\star}(x)}{f(x) |x|^{\big(1-\gamma\big)(n-1)}} & \leq \frac{\Gamma(1-\gamma)^n}{\Gamma((1-\gamma) n)} + \sum_{i=1}^{n-1} \binom{n}{i} M_2^{n-i} \frac{\Gamma(1 - \gamma)^{i}}{\Gamma((1 - \gamma)i)} |x|^{-(1 - \gamma)(n-i)} \\
 & \leq \frac{\Gamma(1-\gamma)^n}{\Gamma((1-\gamma) n)} + \frac{2 M_2 n}{|x|^{1 - \gamma}}  \sum_{i=1}^{n-1} \binom{n-1}{i} M_2^{n-1-i} \Gamma(1 - \gamma)^{i} \\
 & \leq \frac{\Gamma(1-\gamma)^n}{\Gamma((1-\gamma) n)} + \frac{2 M_2 n \big(M_2+\Gamma(1 - \gamma)\big)^{n-1}}{|x|^{1 - \gamma}},
\end{align*}
which is exactly the claimed upper bound.

The first two-sided bound in part (2) follows from Corollary \ref{cor:poiss} and Lemma \ref{lem:direct_1d_gn}. We just observe that after multiplying by $ |x|^{1-\gamma}$  the series appearing in this estimate defines the generalized Bessel function $\phi(1-\gamma,0; \Gamma(1-\gamma)  \lambda|x|^{1-\gamma})$, cf.\ \eqref{eq:GBf}. The second bound holds by the fact that $ 1/\Gamma(\rho) \leq t^{-1} \phi(\rho,0;t) \leq \phi(\rho,0;1)$, $t \in (0,1]$, and the third follows directly from the estimates in Lemma \ref{lem:wright}. 
 
\end{proof}

\subsection{The upper bound for the functions $H_n$ for any dimension}

Recall that the functions $H_n$ are defined in Section \ref{sec:upper_for_hn_Jn}.

\begin{lemma}\label{le:le altH_n}
	Let $f$ be as in \eqref{B} with $g(r) = r^{-\gamma}$ or $g(r) = (1 \vee r)^{-\gamma}$ where $\gamma \in \big[0,\frac{d+1}{2}\big)$. Then
	$$
	H_n(r) \leq \frac{\Gamma\big(\frac{d+1}{2}-\gamma\big)^{n}}{\Gamma\big((\frac{d+1}{2}-\gamma)n\big)}r^{(\frac{d+1}{2}-\gamma)(n-1)}, \quad r \geq 0, \ \ n \in \N.
	$$
\end{lemma}
\begin{proof}
	The argument is inductive. We only need to consider $g(r) = r^{-\gamma}$ with $\gamma \in \big[0,\frac{d+1}{2}\big)$ (the estimate for the second case follows then from the inequality $(1 \vee r)^{-\gamma} \leq r^{-\gamma}$, $r>0$). For $n=1$ this is true as $H_1 \equiv \I_{[0,\infty)}$. We will check the induction step. Clearly, by definition of $H_n$, it is enough to consider $r>2$. By induction hypothesis, we have
	\begin{align*}
	H_{n+1}(r) & =\frac{1}{r^{\frac{d-1}{2}-\gamma}}\int_1^{r-1} (r-\rho)^{\frac{d-1}{2}-\gamma}\rho^{\frac{d-1}{2}-\gamma}H_n(\rho) d\rho \\
	           &\leq \frac{1}{r^{\frac{d-1}{2}-\gamma}} \frac{\Gamma(\frac{d+1}{2}-\gamma)^{n}}{\Gamma\big((\frac{d+1}{2}-\gamma)n\big)}\int_0^{r}(r-\rho)^{\frac{d-1}{2}-\gamma}\rho^{(\frac{d+1}{2}-\gamma)n-1} d\rho .
	\end{align*}
The substitution $\rho = rs$ gives that the expression on the right hand side is equal to 
	$$
	r^{(\frac{d+1}{2}-\gamma)n}\frac{\Gamma(\frac{d+1}{2}-\gamma)^{n}}{\Gamma((\frac{d+1}{2}-\gamma)n)}\int_0^1(1-s)^{\frac{d-1}{2}-\gamma}s^{(\frac{d+1}{2}-\gamma)n-1}ds.
	$$
We see that by \eqref{eq:Beta_def} this is just 
$$
r^{(\frac{d+1}{2}-\gamma)n}\frac{\Gamma(\frac{d+1}{2}-\gamma)^{n}}{\Gamma((\frac{d+1}{2}-\gamma)n)}\frac{\Gamma(\frac{d+1}{2}-\gamma) \Gamma(\Gamma((\frac{d+1}{2}-\gamma)n))}{\Gamma((\frac{d+1}{2}-\gamma)(n+1))},
$$ 
which is exactly what we wanted to get.
\end{proof}

\subsection{The lower bound for the functions $G_n$ for any dimension}

For the definition of the functions $G_n$ we refer the reader to Section \ref{sec:lower_for_gn_Jn}. Recall that for $d=1$ we have $G_n(|x|) = g_n(x)$, $n \in \N$, $x \in \R^d$. These functions were calculated in Section \ref{sec:direct_1d} for the density $f(x) = e^{-m|x|}|x|^{-\gamma}$. Our estimates in this section are less sharp. However, for clarity and completeness of the statements, we decided to not exclude the case $d=1$ from the discussion below. Our results in this section apply to the full range of $d \geq 1$. 

\begin{lemma}\label{lem:estima}
Let $f$ be as in \eqref{B} with $g(r) = r^{-\gamma}$ or $g(r) = (1 \vee r)^{-\gamma}$ where $\gamma \in \big[0,\frac{d+1}{2}\big)$. Then, for every fixed $r_0 > 1$ and $n \in \N$, we have the following statements. 
	\begin{itemize} 
		\item[(1)] Let $r \in (0,r_0]$. If $g(r) = r^{-\gamma}$, then
		$$
		G_{n+1}(r) \geq  C r^{d-\gamma} \int_{0}^{1}(1-u)^{d-1-\gamma} u^{d-1 - \gamma} G_n(ru) du,	
		$$
		while for $g(r) = (1 \vee r)^{-\gamma}$ we have
	$$
		G_{n+1}(r) \geq  C  r^d (1 \vee r)^{-\gamma} \int_{0}^{1}(1-u)^{d-1}u^{d-1} G_n(ru) du,
		$$
		where
		\begin{align} \label{eq:constant_estima_1}
		C = C(r_0):=\begin{cases}
        1 & \text{if\ \ } d=1,\\ 
        \frac{e^{-m r_0}}{d-1} & \text{if\ \ } d > 1.
    \end{cases}
		\end{align}
		\item[(2)] Let $r \geq r_0$. If $g(r) = r^{-\gamma}$, then
		$$
		G_{n+1}(r) \geq  C r^{\frac{d+1}{2}-\gamma} \int_{\frac{1}{r}}^1  (1-u)^{d-1-\gamma} u^{\frac{d-1}{2}- \gamma}G_n(ru) du,		
		$$
		while for $g(r) = (1 \vee r)^{-\gamma}$ we have
	$$
		G_{n+1}(r) \geq C r^{\frac{d+1}{2}-\gamma} \int_{\frac1r}^{1}  (1-u)^{d-1} u^{\frac{d-1}{2}}G_n(ru) du,
		$$ 
		where $C = C(r_0)$ is the constant given by \eqref{eq:constant_estima_1}.
	\end{itemize}
\end{lemma}

\begin{proof}
	Fix $r_0>1$ and suppose first that $g(r) = r^{-\gamma}$. We only consider the case $d >1$. The proof for $d=1$ is just an easy modification and it is much simpler. By the definition of $G_{n+1}(r)$, substitution $\rho = ru$ and the inequality $\sqrt{u \wedge (1-u)} \geq \sqrt{u}\sqrt{1-u}$ valid for all $u \in [0,1]$, we get 
	\begin{align*}
	G_{n+1}(r) & \geq r^{\frac{d+1}{2}-\gamma} \int_{0}^1 \left(\int_0^{\sqrt{r} \sqrt{u} \sqrt{1-u}} e^{-m s^2} s^{d-2}\, ds\right)  (1-u)^{ \frac{d-1}{2}-\gamma} u^{\frac{d-1}{2}-\gamma} G_{n}(ru)  du \\
	                      & \geq \frac{e^{-m r_0}}{(d-1)} r^{d-\gamma} \int_{0}^1  (1-u)^{d-1-\gamma} u^{d-1 - \gamma} G_{n}(ru)  du,
\end{align*}
for $r \in (0,r_0]$. This is the first claimed inequality. For $g(r) = (1 \vee r)^{-\gamma}$, by the same argument, we have for $r \in (0,r_0]$
	\begin{align*}
	G_{n+1}(r) \geq \frac{e^{-m r_0}}{(d-1)} r^{d} \frac{1}{(1 \vee r)^{-\gamma}} \int_{0}^1  (1-u)^{d-1} u^{d-1}(1 \vee r(1-u))^{-\gamma} (1 \vee ru)^{-\gamma}  G_{n-1}(ru)  du.
\end{align*}
Now, since 
\begin{align}\label{eq:aux_lower_1}
(1 \vee r(1-u))^{-\gamma} \geq (1 \vee r)^{-\gamma},  \quad (1 \vee ru)^{-\gamma} \geq (1 \vee r)^{-\gamma}, \quad \text{for} \ u \in [0,1],
\end{align}
we obtain
$$
G_{n+1}(r) \geq \frac{e^{-m r_0}}{(d-1)} r^{d} (1 \vee r)^{-\gamma} \int_{0}^1  (1-u)^{d-1} u^{d-1} G_{n-1}(ru)  du,
$$
which is the second inequality in (1).

In order to show part (2), we use exactly the same argument as for (1). The only difference is that now we integrate over $u \in (1/r,1)$ and therefore 
the inner integral which appears for $d>1$ can be estimated in a little different way:
$$
\int_0^{\sqrt{r} \sqrt{u} \sqrt{1-u}} e^{-m s^2} s^{d-2}\, ds \geq \int_0^{\sqrt{1-u}} e^{-m s^2} s^{d-2}\, ds \geq \frac{e^{-m}}{d-1} (1-u)^{\frac{d-1}{2}} \geq \frac{e^{-mr_0}}{d-1} (1-u)^{\frac{d-1}{2}}.
$$ 
\end{proof}

We are now in a position to give the lower bound of $G_n$ for small arguments. 

\begin{lemma} \label{low-es-small}
	Let $r_0 >1$, $r \in (0,r_0]$, $n \in \N$ and $\gamma \in [0,\frac{d+1}{2})$. Then, for $g(r) = r^{-\gamma}$, 
	$$
	G_n(r) \geq \frac{C^{n-1} \, \Gamma(d-\gamma)^{n}}{\Gamma((d-\gamma)n)} r^{(d-\gamma)(n-1)}
	$$
	and, for $g(r) = (1 \vee r)^{-\gamma}$,  
	$$
	G_n(r) \geq  \frac{C^{n-1} \, \Gamma(d)^{n}}{\Gamma(nd)}  \left(r^d (1 \vee r)^{-\gamma}\right)^{n-1}
	$$
	where the constant $C$ is given by \eqref{eq:constant_estima_1}. 
\end{lemma}	

\begin{proof}
We use induction. For $n=1$ both inequalities hold because $G_1 \equiv \I_{[0,\infty)}$. We have to check the induction step. We first give the proof of the first inequality. Let $r \in (0,r_0]$.
By the first estimate in Lemma \ref{lem:estima} (1) and the induction hypothesis	
\begin{align*} 
		G_{n+1}(r) & \geq  C^n  r^{(d-\gamma)n} \frac{\Gamma(d-\gamma)^{n}}{\Gamma((d-\gamma)n)} \int_{0}^{1} (1-u)^{d-\gamma-1}u^{d - \gamma -1} u^{(d-\gamma)(n-1)}du  \\
		           & =  C^n  r^{(d-\gamma)n} \frac{\Gamma(d-\gamma)^{n}}{\Gamma((d-\gamma)n)} \int_{0}^{1} (1-u)^{d-\gamma-1}u^{(d-\gamma)n-1}du
	\end{align*}
Observe that the expression on the right hand side can be rewritten as
$$
 C^n  r^{(d-\gamma)n} \frac{ \Gamma(d-\gamma)^n}{\Gamma((d-\gamma)n)} \frac{\Gamma((d-\gamma)n)\Gamma(d-\gamma)}{\Gamma((d-\gamma)(n+1))} = C^n  r^{(d-\gamma)n} \frac{ \Gamma(d-\gamma)^{n+1}}{\Gamma((d-\gamma)(n+1))}.
$$ 
This is exactly what we wanted to get. 

The proof of the induction step for the second estimate is similar:\ we start with the second inequality in Lemma \ref{lem:estima} (1), use the induction hypothesis, and then apply the inequality
$$
(1 \vee ru)^{-\gamma} \geq (1 \vee r)^{-\gamma}, \quad u \in [0,1],
$$
getting the desired bound as above.
\end{proof}

The next corollary follows directly from Lemma \ref{low-es-small} and the inequality $r^d \geq r^{\frac{d+1}{2}}$, $r \geq 1$.

\begin{corollary}\label{col:col1}
	Let $r_0 >1$, $r \in [1,r_0]$, $n \in \N$ and $\gamma \in [0,\frac{d+1}{2})$. If $g(r) = r^{-\gamma}$, then
	$$
	G_n(r) \geq  \frac{C^{n-1} \, \Gamma(d-\gamma)^{n}}{\Gamma((d-\gamma)n)} r^{\big(\frac{d+1}{2}-\gamma\big)(n-1)},
	$$
and if $g(r) = (1 \vee r)^{-\gamma}$, then
	$$
	G_n(r) \geq  \frac{C^{n-1} \, \Gamma(d)^{n}}{\Gamma(dn)} r^{\big(\frac{d+1}{2}-\gamma\big)(n-1)},
	$$
		where the constant $C$ is given by \eqref{eq:constant_estima_1}. 
\end{corollary} 

The following lemma concerning the incomplete beta function is critical for our estimates below. It is crucial that the range of $r$ does not depend on the parameter $a$. 

\begin{lemma}\label{lem:beta:quotient}
	 Fix $a_0 \in (0,1]$ and $b > 0$. Then for every $a \geq a_0$   and 
	\begin{align} \label{eq:range_x} 
	r \geq r_0 :=  \left(\frac{8 \ceil{b}^{\floor{b}} }{a_0^{1-\{b\}}}\right)^{\frac{1}{a_0}}  \big(2^{\frac{1}{a_0}} e^{\frac{1}{e}}\big)^{2\floor{b}}
	\end{align}
	we have
	$$
	2 B_{\frac{1}{r}}(a, b) \leq  B(a,b).
	$$ 
\end{lemma}

\begin{proof}
Fix  $a_0 \in (0,1]$ and $b > 0$, and consider $a \geq a_0$ and  $r$ satisfying \eqref{eq:range_x}. 
	Using representation \eqref{eq:iBeta_hypergeom} and the definition of the (Gauss) hypergeometric function, we get
	\begin{align*} 
		\frac{B_{\frac{1}{r}}(a, b)}{B(a,b)} & = \frac{\Gamma(a+b)}{\Gamma(a)\Gamma(b)} \frac{\left(\frac{1}{r}\right)^a (1-\frac{1}{r})^b}{a}F\left(a+b,1,a+1; \frac{1}{r}\right) \\ 
	                                                         & = \frac{\Gamma(a+b)}{\Gamma(a)\Gamma(b)} \frac{\left(\frac{1}{r}\right)^a (1-\frac{1}{r})^b}{a} \frac{\Gamma(a +1)}{\Gamma(a + b)\Gamma(1)}\sum_{s=0}^{\infty} \frac{\Gamma(a+ b +s)\Gamma(1+s)}{\Gamma(a + 1+s)s!}\left(\frac{1}{r}\right)^s, \nonumber
	\end{align*}
which, after some trivial cancellations, is equal to
	\begin{align} \label{eq:quotient}
 \left(\frac{1}{r}\right)^a \left(1-\frac{1}{r}\right)^b \frac{1}{\Gamma(b)} \sum_{s=0}^\infty \frac{\Gamma(a+b+s)}{\Gamma(a+1+s)} \left(\frac{1}{r}\right)^s.
	\end{align}
 Set $\{b\} = b - \floor{b}$. The property $\Gamma(r+1) = r \, \Gamma(r)$ and Gautschi's inequality \cite[5.6.4]{NIST:DLMF} in the form
$$
\frac{\Gamma(x+y)}{\Gamma(x+1)} <  \frac{1}{x^{1-y}}, \quad x >0, \ y \in (0,1),
$$
yield
	\begin{align}\label{eq:aux_gamma}
		\frac{\Gamma(a+b+s)}{\Gamma(a+1+s)} \leq 	\frac{\Gamma(a+\{b\}+s)}{\Gamma(a+1+s)}(a+s+\ceil{b})^{\floor{b}} \leq \frac{1}{(a+s)^{1-\{b\}}}(a+s+\ceil{b})^{\floor{b}} , 
	\end{align}
	for every $s\in \N\cup\{0\}$, and further,  by elementary inequality $a+s+\ceil{b} \leq (a+1)(s+1)\ceil{b}$, we get
	\begin{align*}
	 	(a+s+\ceil{b})^{\floor{b}} \leq ({(a+1)(s+1) \ceil{b}})^{\floor{b}}. 
	\end{align*} 
	Hence,
	\begin{align*}
		\frac{\Gamma(a+b+s)}{\Gamma(a+1+s)} & \leq \frac{(a+1)^{\floor{b}}(s+1)^{\floor{b}} \ceil{b}^{\floor{b}}}{a_0^{1-\{b\}}}, 
		\quad s\in \N\cup\{0\}.
	\end{align*} 
	Now, observe that under \eqref{eq:range_x} we have
	\begin{equation}\label{lem1:condition1}
			 \sup\limits_{s \in \N}{(\sqrt[s]{s+1})\big)^{2\floor{b}}} \leq  \big(2\sup\limits_{s \in \N}{\sqrt[s]{s})\big)^{2\floor{b}}} \leq \big(2e^{\frac{1}{e}}\big)^{2\floor{b}} \leq r_0  \leq r. 
	\end{equation}
	Consequently, $(s+1)^{\floor{b}} \leq r^{\frac{s}{2}}$, $s \in \N$ (for $s=0$ this inequality is trivial), which gives
	\begin{align*}
		\sum_{s=0}^\infty \frac{\Gamma(a+b+s)}{\Gamma(a+1+s)} \left(\frac{1}{r}\right)^s \leq  \frac{\ceil{b}^{\floor{b}} (a +1 )^{\floor{b}}}{a_0^{1-\{b\}}} \sum_{s=0}^\infty \left(\frac{1}{r}\right)^{\frac{s}{2}} \leq \frac{2 \ceil{b}^{\floor{b}} (a +1 )^{\floor{b}}}{a_0^{1-\{b\}}}
	\end{align*}
	 (the last inequality follows from the fact that $\sum_{s=0}^\infty r^{-\frac{s}{2}} \leq 2$  as $r \geq r_0 \geq 4$). 
	Furthermore, recall that $\Gamma(b) \geq \frac{1}{2}$, see \eqref{eq:gamma_est}. 
	Therefore, coming back to \eqref{eq:quotient}, we get
	\begin{align*}
		\frac{B_{\frac{1}{r}}(a, b)}{B(a,b)} & \leq \left(\frac{1}{r}\right)^a \frac{1}{\Gamma(b)}  \frac{2 \ceil{b}^{\floor{b}} (a +1 )^{\floor{b}}}{a_0^{1-\{b\}}} \\ & \leq \left(\frac{1}{r}\right)^a \frac{4 \ceil{b}^{\floor{b}} (a +1 )^{\floor{b}}}{a_0^{1-\{b\}}}
		      = \frac{1}{2}\left(\frac{1}{r} \left(\frac{8 \ceil{b}^{\floor{b}} (a +1 )^{\floor{b}}}{a_0^{1-\{b\}}}\right)^{\frac{1}{a}}\right)^a.
	\end{align*}

Finally,  since $(1+a)^{1/a} \leq 2^{1/a_0} e^{1/e}$, $a \geq a_0$ (cf.\ \eqref{lem1:condition1}), l by \eqref{eq:range_x} we have
	\begin{align*}
 \left(\frac{8 \ceil{b}^{\floor{b}} (a +1 )^{\floor{b}}}{a_0^{1-\{b\}}}\right)^{\frac{1}{a}}=\left(\frac{8 \ceil{b}^{\floor{b}} }{a_0^{1-\{b\}}}\right)^{\frac{1}{a}} (a +1 )^{\frac{\floor{b}}{a}} \leq \left(\frac{8 \ceil{b}^{\floor{b}} }{a_0^{1-\{b\}}}\right)^{\frac{1}{a_0}}  \big(2^{\frac{1}{a_0}} e^{\frac{1}{e}}\big)^{2\floor{b}} = r_0 \leq r.
	\end{align*}
This leads us to a conclusion
	\begin{align*}
		\frac{B_{\frac{1}{r}}(a, b)}{B(a,b)} \leq \frac{1}{2}.
	\end{align*}
\end{proof}

We are now ready to give a final lemma in this section. 

\begin{lemma}\label{lem:G_lower} 
Let $r \geq 1$, $n \in \N$ and $\gamma \in [0,\frac{d+1}{2})$. If $g(r) = r^{-\gamma}$, then
	$$
	G_n(r) \geq \left(\frac{C}{2}\right)^{n-1}  \frac{  \Gamma(d-\gamma)^{n}}{\Gamma((d-\gamma)n)} r^{\big(\frac{d+1}{2}-\gamma\big)(n-1)},
	$$
where $C = C(r_0)$ is the constant in \eqref{eq:constant_estima_1} with $r_0$ given by \eqref{eq:range_x} for $ a_0 = b = d -\gamma$. 

If $g(r) = (1 \vee r)^{-\gamma}$, then
	$$
	G_n(r) \geq \left(\frac{C}{2}\right)^{n-1} \frac{  \Gamma(d)^{n}}{\Gamma(dn)}  r^{\big(\frac{d+1}{2}-\gamma\big)(n-1)},
	$$
where $C = C(r_0)$ is the constant in \eqref{eq:constant_estima_1} with $r_0$ given by \eqref{eq:range_x} for $a_0 = 1$ and $b = d$.
\end{lemma}

\begin{proof} 
Since $G_1 \equiv \I_{[0,\infty)}$, we only need to check the induction step. We consider the first assertion. Let $r_0$ be as above and let $r \geq r_0$.  From the first estimate of Lemma \ref{lem:estima} (2), we get
$$
		G_{n+1}(r) \geq C r^{\frac{d+1}{2}-\gamma} \int_{\frac1r}^{1}  (1-u)^{d-\gamma-1} u^{\frac{d-1}{2} - \gamma}  G_n(ru) du,
$$
where $C = C(r_0)$ is the constant defined in \eqref{eq:constant_estima_1}. Using the induction hypothesis, we further obtain 
\begin{align*}
		G_{n+1}(r) & \geq  \frac{C^n}{2^{n-1}}  \frac{\Gamma(d-\gamma)^{n}}{\Gamma((d-\gamma)n)}  r^{\big(\frac{d+1}{2}-\gamma\big)n} \int_{\frac1r}^{1}  (1-u)^{d-\gamma-1} u^{\frac{d-1}{2} - \gamma}  u^{\big(\frac{d+1}{2}-\gamma\big)(n-1)}  du \\
	            &  = \frac{C^n}{2^{n-1}}  \frac{\Gamma(d-\gamma)^{n}}{\Gamma((d-\gamma)n)} r^{\big(\frac{d+1}{2}-\gamma\big)n} \int_{\frac1r}^{1} (1-u)^{d-\gamma-1} u^{\big(\frac{d+1}{2}-\gamma\big)n -1} du \\
							&  \geq \frac{C^n}{2^{n-1}} \frac{\Gamma(d-\gamma)^{n}}{\Gamma((d-\gamma)n)} r^{\big(\frac{d+1}{2}-\gamma\big)n} \int_{\frac1r}^{1}  (1-u)^{d-\gamma-1} u^{\big(d-\gamma\big)n-1}  du,
\end{align*}
which can be rewritten as
\begin{align*}
  \frac{C^n}{2^{n-1}} \frac{\left(\Gamma(d - \gamma)\right)^{n-1}}{\Gamma((d- \gamma)n)} r^{\big(\frac{d+1}{2}-\gamma\big)n} \left(B\big((d-\gamma)n,d-\gamma\big) - B_{\frac{1}{r}}\big((d-\gamma)n,d-\gamma\big)\right).
\end{align*}
Lemma \ref{lem:beta:quotient} gives that for $r \geq r_0$ we have
$$
B\big((d-\gamma)n,d-\gamma\big) - B_{\frac{1}{r}}\big((d-\gamma)n,d-\gamma\big)\geq \frac{1}{2} B\big((d-\gamma)n,d-\gamma\big),
$$
which implies that the above expression is bigger than or equal to
$$
\left(\frac{C}{2}\right)^n \frac{\Gamma(d-\gamma)^{n}}{\Gamma((d-\gamma)n)} r^{\big(\frac{d+1}{2}-\gamma\big)n}B\big((d-\gamma)n,d-\gamma\big)  =\left(\frac{C}{2}\right)^n \frac{ \Gamma(d-\gamma)^{n+1}}{\Gamma((d-\gamma)(n+1))} r^{\big(\frac{d+1}{2}-\gamma\big)n},
$$
as long as $r \geq r_0$. For $r \in [1,r_0]$ the same bound follows directly from the first estimate of Corollary \ref{col:col1}. This completes the proof of the first inequality. 

The proof of the second assertion follows the same steps. We apply the second estimate of Lemma \ref{lem:estima} (2) and Lemma \ref{lem:beta:quotient} to check the induction step for $r \geq r_0$.  The corresponding bound for $r \in [1,r_0]$ follows from the second estimate of Corollary \ref{col:col1}. 
\end{proof}

\subsection{Estimates of convolutions $f^{n\star}$ and densities $p_{\lambda}$ for any dimension}

 We are now in a position to give estimates of $f^{n\star}$ and $p_{\lambda}(x)$ for $f$'s as in \eqref{eq:exp_dens} and \eqref{eq:exp_dens_cut} with $\gamma \in \big[0,\frac{d+1}{2}\big)$. We focus on the case $|x| \geq 1$, $\lambda>0$ and first give general estimates of $p_{\lambda}$ in terms of generalized Bessel function. Recall that estimates of $p_{\lambda}$ for $|x| < 1$ and $\lambda>0$ are given in Corollary \ref{cor:poiss} and, therefore, we do not discuss that case in this section. 

\begin{theorem}\label{th:poisson} 
	Let $f$ be as in \eqref{B} with $g(r) = r^{-\gamma}$ or $g(r) = (1 \vee r)^{-\gamma}$ where $\gamma \in \big[0,\frac{d+1}{2}\big)$. 
	We set 
	$$ 
	D_1 = C/2, \ \ D_2 = 1 \ \text{if} \ d=1, \  \text{and} \ \  D_1=CM_4/2, \ D_2 = M_3 \ \text{if} \ d>1,
	$$
	where $C$ is the constant in Lemma \ref{lem:G_lower}, and
	$$
	\quad \begin{cases}
        \rho_1 = d-\gamma \ \text{and} \ \kappa_1 =  D_1  \Gamma(d-\gamma)  & \text{for\ \ } g(r) = r^{-\gamma},\\ 
        \rho_1 = d \ \ \ \ \ \ \text{and} \ \kappa_1 =  D_1  \Gamma(d)      & \text{for\ \ } g(r) = (1 \vee r)^{-\gamma},
    \end{cases}
	$$
	$$
	 \quad \rho_2 = \frac{d+1}{2}-\gamma, \quad \kappa_2 =  D_2  \Gamma\left(\frac{d+1}{2}-\gamma\right).
	$$
	Then we have the following estimates.
	\begin{itemize}
	\item[(1)] For $|x| \geq 1$ and $n \in \N$, 
	$$
	D_1^{n-1} \frac{\Gamma(\rho_1)^n}{\Gamma(\rho_1n)} \leq \frac{f^{n\star}(x)}{f(x) |x|^{\big(\frac{d+1}{2}-\gamma\big)(n-1)}} \leq D_2^{n-1} \frac{\Gamma(\rho_2)^n}{\Gamma(\rho_2 n)} + \frac{2M_2 n (M_2+\kappa_2)^{n-1}}{D_2 |x|^{\frac{d+1}{2}-\gamma}}. 
	$$
In particular,
	\begin{align*}
	D_1^{n-1} \frac{\Gamma(\rho_1)^n}{\Gamma(\rho_1n)}  \leq \liminf_{|x| \to \infty} &\frac{f^{n\star}(x)}{f(x) |x|^{\big(\frac{d+1}{2}-\gamma\big)(n-1)}} \\ & \leq \limsup_{|x| \to \infty} \frac{f^{n\star}(x)}{f(x) |x|^{\big(\frac{d+1}{2}-\gamma\big)(n-1)}} \leq D_2^{n-1} \frac{\Gamma(\rho_2)^n}{\Gamma(\rho_2 n)}
	\end{align*}
	\item[(2)] If $|x| \geq 1$ and $\lambda>0$, then 
	$$
	\phi\big(\rho_1,0;\kappa_1 \lambda |x|^{\frac{d+1}{2}-\gamma}\big) \leq \frac{p_{\lambda}(x)}{ e^{- \lambda \norm{f}_1} e^{-m|x|} |x|^{-\frac{d+1}{2}}} \leq e^{M_2\lambda} \phi\big(\rho_2,0;\kappa_2 \lambda |x|^{\frac{d+1}{2}-\gamma}\big).
	$$
\end{itemize}
\end{theorem}

\begin{proof}
The lower bound in (1) follows from Lemma \ref{lem:lower_gen}, Theorem \ref{th:th4} and Lemma \ref{lem:G_lower} . 
 For the proof of the upper estimate we observe that by Theorems \ref{th:th1} and \ref{th:th2}, and Lemma \ref{le:le altH_n} we have
\begin{align*}
	\frac{f^{n\star}(x)}{f(x) |x|^{\big(\frac{d+1}{2}-\gamma\big)(n-1)}} & \leq D_2^{n-1} \frac{\Gamma(\rho_2)^n}{\Gamma(\rho_2 n)} + \sum_{i=1}^{n-1} \binom{n}{i} M_2^{n-i}  D_2^{i-1}  \frac{\Gamma(\frac{d+1}{2} - \gamma)^{i}}{\Gamma((\frac{d+1}{2} - \gamma)i)} |x|^{-(\frac{d+1}{2} - \gamma)(n-i)}.
	\end{align*}
 Now, since $\Gamma(r) \geq 1/2$ for $r>0$, $\binom{n}{i} \leq n \binom{n - 1}{i}$ and $|x| \geq 1$, we get
\begin{align*}
\sum_{i=1}^{n-1} \binom{n}{i} M_2^{n-i} D_2^{i-1} & \frac{\Gamma(\frac{d+1}{2} - \gamma)^{i}}{\Gamma((\frac{d+1}{2} - \gamma)i)} |x|^{-(\frac{d+1}{2} - \gamma)(n-i)}\\
 & \leq \frac{2 M_2 n}{D_2|x|^{\frac{d+1}{2} - \gamma}}  \sum_{i=1}^{n-1} \binom{n-1}{i} M_2^{n-1-i} \left(D_2\Gamma(\frac{d+1}{2} - \gamma)\right)^{i} \\
 & \leq \frac{2 M_2 n \big(M_2+D_2\Gamma(\frac{d+1}{2} - \gamma)\big)^{n-1}}{D_2|x|^{\frac{d+1}{2} - \gamma}},
\end{align*}
which gives the claimed upper bound.

We are left to show (2). We first establish the upper bound. By Corollary \ref{cor:poiss}, Theorem \ref{th:th2} and Lemma \ref{le:le altH_n}, we get
	\begin{align}\label{plamd1}
		\frac{p_{\lambda}(x)}{ e^{- \norm{f}_1\lambda} e^{-m|x|} |x|^{-\frac{d+1}{2}}} \leq  e^{M_2\lambda} |x|^{\frac{d+1}{2}-\gamma} \sum_{n=1}^{\infty} \frac{\lambda^{n} h_n(x)}{n!} 
		\leq e^{M_2\lambda} \sum_{n=1}^{\infty}  \frac{(  D_2  \Gamma(\frac{d+1}{2}-\gamma) \lambda |x|^{\frac{d+1}{2}-\gamma})^{n}}{\Gamma((\frac{d+1}{2}-\gamma)n) \; n!},
	\end{align}
	where the last series can be identified with $\phi\big(\rho_2,0;\kappa_2 \lambda |x|^{\frac{d+1}{2}-\gamma}\big)$, giving the claimed upper estimate in (2). 
	
The lower bound in part (2) follows directly from the definition of $p_{\lambda}(x)$ in \eqref{eq:compound_dens} and the lower estimate in part (1):
\begin{equation}\label{plambd2} 
		\frac{p_{\lambda}(x)}{ e^{- \norm{f}_1\lambda} e^{-m|x|} |x|^{-\frac{d+1}{2}}}  \geq  \frac{f(x)\sum\limits_{n=1}^{\infty} \tfrac{ D_1^n\Gamma(\rho_1)^n}{\Gamma(n\rho_1)n!}\left(\lambda|x|^{\tfrac{d+1}{2}-\gamma}\right)^n}{e^{-m|x|}|x|^{-\gamma}} \geq  \phi\big(\rho_1,0;\kappa_1 \lambda |x|^{\frac{d+1}{2}-\gamma}\big).
\end{equation}
\end{proof}

Finally, using Lemma \ref{lem:wright}, we get more explicit estimates of the densities $p_{\lambda}(x)$ for $|x| \geq 1$ and $\lambda>0$. We observe two different regimes.

\begin{corollary} \label{cor:final}
Let $f$ be as in \eqref{B} with $g(r) = r^{-\gamma}$ or $g(r) = (1 \vee r)^{-\gamma}$ where $\gamma \in \big[0,\frac{d+1}{2}\big)$. Then there are constants $E_1-E_6$ such that for every $|x| \geq 1$ and $\lambda>0$ we have the following estimates. 
\begin{itemize}
\item[(1)] If $\lambda |x|^{\frac{d+1}{2}-\gamma} \leq 1$, then
  $$
			E_1 \leq \frac{p_{\lambda}(x)}{ \lambda e^{-\lambda \norm{f}_1 } e^{-m|x|}  |x|^{-\gamma}}  \leq E_2.
  $$
\item[(2)] If $\lambda |x|^{\frac{d+1}{2}-\gamma} \geq 1$, then
	\begin{equation*}
		E_3 \exp\left( E_4 \big(\lambda |x|^{\frac{d+1}{2}-\gamma}\big)^{\frac{1}{\rho_1 + 1}}\right) \leq \frac{p_{\lambda}(x) }{ e^{-\lambda \norm{f}_1 } e^{-m|x|} |x|^{-\frac{d+1}{2}}} \leq  E_5 e^{\lambda M_2} \exp\left(E_6 \big(\lambda |x|^{\frac{d+1}{2}-\gamma}\big)^{\frac{1}{\rho_2  + 1}}\right), 
	\end{equation*}
	where $\rho_1$ and $\rho_2$ are those of Theorem \ref{th:poisson}.
\end{itemize}
\end{corollary}

\begin{proof}
 In order to get estimates in part (1) it is convenient to use directly estimates \eqref{plamd1}, \eqref{plambd2}. We see that one has \ $E_1 = D_1$ (we only keep the first term of the series) and $E_2=e^{M_2}  \phi\big(\rho_2,0;\kappa_2\big)$ (we use the estimate $\lambda |x|^{(d+1)/2-\gamma} \leq 1$ under the series;  in particular $\lambda \leq 1$). Part (2) is a direct consequence of estimates in Theorem \ref{th:poisson} (2) and Lemma \ref{lem:wright} above.
\end{proof}

\end{document}